\documentclass[12pt]{article}

\usepackage[english]{babel}
\usepackage[margin=1in]{geometry}
\usepackage{authblk}

\usepackage{amsmath}
\usepackage{amsthm}
\usepackage{amssymb}
\usepackage{hyperref}
\usepackage{microtype}

\usepackage[
backend=biber,
style=alphabetic,
]{biblatex}
\usepackage{amsmath}
\usepackage{graphicx}

\newcommand{\id}{\mathrm{id}}

\DeclareMathOperator{\fix}{Fix}
\DeclareMathOperator{\hol}{Hol}

\DeclareMathOperator{\diam}{diam}
\DeclareMathOperator{\lip}{Lip}

\theoremstyle{plain}
\newtheorem{theorem}{Theorem}

\newtheorem{lemma}[theorem]{Lemma}
\newtheorem{sublemma}[theorem]{Sublemma}
 
\newtheorem{cor}[theorem]{Corollary}

\newtheorem*{acknowledgements}{Acknowledgments}

\numberwithin{theorem}{section}
\numberwithin{equation}{section}

\theoremstyle{definition}
\newtheorem{definition}{Definition}

\theoremstyle{remark}
\newtheorem{rem}{Remark}

\title{Finite Periodic Data Rigidity For Two-Dimensional Area-Preserving Anosov Diffeomorphisms}
\author{Thomas Aloysius O'Hare}

\date{September 2024}

\begin{document}
\maketitle
\begin{abstract}
    \noindent Let $f,g$ be $C^2$ area-preserving Anosov diffeomorphisms on $\mathbb{T}^2$ which are topologically conjugate by a homeomorphism $h$ ($hf=gh$). We assume that the Jacobian periodic data of $f$ and $g$ are matched by $h$ for all points of some large period $N\in\mathbb{N}$. We show that $f$ and $g$ are ``approximately smoothly conjugate." That is, there exists a $C^{1+\alpha}$ diffeomorphism $\overline{h}_N$ such that $h$ and $\overline{h}_N$ are $C^0$ exponentially close in $N$, and $f$ and $f_N:=\overline{h}_N^{-1}g\overline{h}_N$ are $C^1$ exponentially close in $N$. Moreover, the rates of convergence are uniform among different $f,g$ in a $C^2$ bounded set of Anosov diffeomorphisms. The main idea in constructing $\overline{h}_N$ is to do a ``weighted holonomy" construction, and the main technical tool in obtaining our estimates is a uniform effective version of Bowen's equidistribution theorem of weighted discrete orbits to the SRB measure.
\end{abstract}

\section{Introduction}
\label{sec: Introduction}

Let $f,g:M\rightarrow M$ be two transitive Anosov diffeomorphisms on a compact manifold $M$ which are topologically conjugate by a homeomorphism $h$; that is, $hf=gh$. We say that $f$ and $g$ have matching periodic data (under the conjugacy $h$), if for every $n\in\mathbb{N}$ and every $p\in\fix(f^n)$, the linear maps $D_pf^n$ and $D_{h(p)}g^n$ are conjugate. If we assume that our conjugacy $h$ were $C^1$ and differentiate the conjugacy equation at the point $p$, we would see that $f$ and $g$ would have matching periodic data. In other words, the matching of the periodic data is a necessary condition for the conjugacy to be at least $C^1$. By taking the stable and unstable Jacobians of the differentiated conjugacy equation at the point $p$, we arrive at the weaker (but simpler to check) condition that 
\begin{equation}
\label{eqn: Jacobian Periodic Data Matching Condition}
    D_uf^n(p)=D_ug^n(h(p)) \hspace{5mm}\text{and}\hspace{5mm} D_sf^n(p)=D_sg^n(h(p)),
\end{equation}
where $D_u$ and $D_s$ denote the unstable and stable Jacobians, respectively. Thus the matching of the Jacobian periodic data in the sense of (\ref{eqn: Jacobian Periodic Data Matching Condition}) is a necessary condition for the conjugacy $h$ to be at least $C^1$. When (\ref{eqn: Jacobian Periodic Data Matching Condition}) is sufficient for $h$ to be $C^1$ or better, then we say that $f$ is \emph{periodic data rigid}.

Which Anosov diffeomorphisms are periodic data rigid has been extensively studied. For low-dimensional Anosov diffeomorphisms the answer is quite well understood. In dimension one, it was proved by Shub and Sullivan \cite{shub_sullivan_1985} that expanding maps on the circle are periodic data rigid; this was considerably generalized by Martens and de Melo \cite{Martens-Marco} who proved that $C^r$ Markov maps on the circle are periodic data rigid. In dimension two, it follows by the work of de la Llave, Marco, and Moriy\'on (\cite{dlLM}, \cite{delaLlave}), that all Anosov diffeomorphisms are periodic data rigid.

The present paper investigates the regularity of the conjugacy $h$ if we assume that the periodic data matches in the sense of (\ref{eqn: Jacobian Periodic Data Matching Condition}) for only finitely many periodic points when $f$ and $g$ are both area-preserving. Of course, if the periodic data fails to match up for just a single periodic orbit, then the conjugacy $h$ cannot be $C^1$. Instead, we show that as the largest number $N$ for which all orbits of period $N$ satisfy (\ref{eqn: Jacobian Periodic Data Matching Condition}) increases, we can find, in a uniform sense, a new $C^{1+\alpha}$ diffeomorphism that acts as smooth conjugacy between nearby Anosov systems:
 

\begin{theorem}
\label{Thm:Main Theorem}
    Let $\mathcal{U}\subset A^2(\mathbb{T}^2)$ be a closed and bounded subset of the set of all $C^2$ area-preserving Anosov diffeomorphisms in the same homotopy class as a fixed linear hyperbolic automorphism $F_L$. Then there exist constants $0<\alpha<1$ $C_0,C_1>0$ and $0<\lambda_0,\lambda_1<1$ depending only on the set $\mathcal{U}$ such that the following holds: If $f,g\in\mathcal{U}$ are conjugated by a homeomorphism $h$ homotopic to the identity ($hf=gh$), and if there exists a natural number $N$ such that for every point $p\in\fix(f^N)$ we have 
    $(D_uf^N)(p)=(D_ug^N)(h(p))$, and $(D_sf^N)(p)=(D_sg^N)(h(p))$,
    then there exists a $C^{1+\alpha}$ diffeomorphism $\overline{h}_N$ of $\mathbb{T}^2$ such that $d_{C^0}(h,\overline{h}_N)$, $d_{C^0}(h^{-1},\overline{h}_N^{-1})<C_0\lambda_0^N$, and $d_{C_1}(f,f_N)<C_1\lambda_1^N$, where $f_N:=\overline{h}_N^{-1}g\overline{h}_N$.
\end{theorem}
\begin{rem}
\label{rem: area-Preserving Hypothesis}
    The area-preserving assumption is not needed to construct the diffeomorphism $\overline{h}_N$. Indeed, for dissipative Anosov diffeomorphisms, there is an $\varepsilon>0$, depending only on $\mathcal{U}$, such that our construction will yield a $C^{1+\varepsilon}$ conjugacy $\overline{h}_N$. However, at the present time we do not know how to perform the estimates in Section \ref{sec: Proof the Main Theorem the C^0 Estimates} without the area-preserving assumption.
\end{rem}

The uniformity statement of Theorem \ref{Thm:Main Theorem} deserves emphasis. For any fixed $f,g\in\mathcal{U}$, the number $N$ of the highest period for which (\ref{eqn: Jacobian Periodic Data Matching Condition}) is satisfied is fixed, and we can trivially satisfy the conclusions of the theorem with any $C^{1+\alpha}$ diffeomorphism $\overline{h}_N$ by taking the constants $C_0$ and $C_1$ sufficiently large. However, Theorem \ref{Thm:Main Theorem} says that these constant, as well as the rates of convergence, can be chosen the same for every pair $f,g\in\mathcal{U}$. In particular, as we choose $f,g\in\mathcal{U}$ such that (\ref{eqn: Jacobian Periodic Data Matching Condition}) is satisfied for larger and larger values of $N$, we will genuinely get conjugacies $\overline{h}_N$ and Anosov diffeomorphisms $\overline{f}_N$ closer and closer to the original $h$ and $f$ in the respective topologies. Therefore a considerable amount of the technical difficulties we encounter will be related to ensuring that our estimates are uniform on $\mathcal{U}$. 

The motivation behind our construction of $\overline{h}_N$ comes from de la Llave's proof of periodic data rigidity in dimension two \cite{delaLlave}. Here smoothness of $h$ is established separately for the restrictions to the unstable and stable manifolds by showing that when condition (\ref{eqn: Jacobian Periodic Data Matching Condition}) is satisfied for every periodic orbit, $h$ may be expressed in terms of the smooth densities of the conditional measures of the SRB measures of $f$ and $g$ on unstable leaves (and stable leaves by considering inverses). Along a fixed unstable leaf, this expression in terms of smooth densities is well-defined even if (\ref{eqn: Jacobian Periodic Data Matching Condition}) fails for some periodic orbit, but the expression we get will not equal the restriction of $h$ to this leaf. By minimality of the unstable foliation, this defines a function on a dense subset of $\mathbb{T}^2$ which is $C^2$ in the unstable direction. Unfortunately, this function will in general be wildly discontinuous in the transverse direction and hence will not extend to a continuous function on $\mathbb{T}^2$.

We are now ready to discuss the proof of Theorem \ref{Thm:Main Theorem}, which broadly speaking will take place in three steps. First is the construction of the new conjugacy $\overline{h}_N$. Motivated by the discussion in the preceding paragraph, we define a function $h_N$ in the manner described above in terms of smooth densities on a compact segment of an unstable leaf of $f$ which we denote by $A^u_f$. Then we extend the domain of definition of $h_N$ from $A^u_f$ to all of $\mathbb{T}^2$ using the stable holonomies of $f$ and $g$. Given any point $x\in\mathbb{T}^2/A^u_f$, there are two ``directions" we can travel along the stable leaves to get to $A^u_f$, and hence two choices for holonomies. Our main trick in the construction to ensure continuity is to consider both possible choices of holonomies and to then take an appropriately weighted average of the outcomes. This construction heavily relies on the fact that in low dimensions the stable holonomy is always $C^{1+\varepsilon}$, where in higher dimensions such high regularity of the holonomy is rare. The result of this weighted-holonomy construction is a homeomorphism $h_N:\mathbb{T}^2\rightarrow\mathbb{T}^2$ which is $C^{1+\alpha}$ when restricted to any unstable leaf. By repeating the same construction along the stable foliation, we obtain the desired $C^{1+\alpha}$ diffeomorphism $\overline{h}_N$. 

The second main step of the proof is to estimate the $C^0$ distance between $h$ and $\overline{h}_N$. The key ideas are similar to those in \cite{FDR1} though more complicated in two-dimensions. We establish the exponential rate of convergence from the assumption on finite periodic data matching by proving that the periodic points effectively equidistribute to the SRB measure, with an exponential rate for Lipschitz observables. This result, which is of independent interest, generalizes Bowen's equidistribution of equilibrium states \cite{Bowen_1974} in the special case of the SRB measure. Finally, the third step of the proof is to establish the $C^1$ estimate between $f$ and $\overline{f}_N$. This will follow quite easily from the first two steps by interpolating between the $C^0$ distance and a uniform $C^{1+\alpha}$ bound.

The paper is organized as follows. In Section \ref{sec: Preliminaries}, we will recall standard definitions, notation, and results we will need. A few standard results are stated in a way better suited for the uniformity arguments we will use. Section \ref{sec: Proof of the Main Theorem: The Construction} is dedicated to the construction of the conjugacy $\overline{h}_N$. As mentioned before, uniformity of the construction is our primary concern, so a considerable amount of technical computations are necessary. To help the reader, this section is further broken into three subsections. In Section \ref{sec: Uniformity of the Initial Conjugacy}, we prove uniformity of the original conjugacy $h$ for any conjugate pair $f,g\in\mathcal{U}$. The H\"older exponent and seminorm of this conjugacy will appear several times in the construction and estimates so it is important that we have uniform bounds. Section \ref{sec: The Construction} describes the construction of $\overline{h}_N$. Along the way we state several lemmas which state that at each given step of the construction, the resulting function is uniformly bounded in the appropriate norm. These lemmas, culminating in Theorem \ref{Thm:Uniform Conjugacy Bounds}, are the heart of Section \ref{sec: Proof of the Main Theorem: The Construction}; they are also the most technically difficult parts of the paper. For this reason, the proofs of those lemmas are collected in Section \ref{sec: Uniformity of the New Conjugacy}. For a first reading, it may be advisable to skip Sections \ref{sec: Uniformity of the Initial Conjugacy} and \ref{sec: Uniformity of the New Conjugacy} and think about the construction of $\overline{h}_N$ for just a fixed pair $f,g\in\mathcal{U}$ without being bogged down by the technical uniformity arguments.

In Section \ref{sec: Proof the Main Theorem the C^0 Estimates} we prove the stated estimates on $d_{C^0}(h,\overline{h}_N)$. This will be done using a effective equidistribution argument (Theorem \ref{Thm:Effective Equidistribution}). The proof of the effective equidistribution is largely similar to the argument presented in \cite{FDR1}, though more care is needed with over counting orbits when we lift to the symbolic setting using Markov partitions. The proof of Theorem \ref{Thm:Effective Equidistribution} will be postponed to Section \ref{sec: Proof of EE}. Finally, Section \ref{sec: Proof of the Main Theorem: The C1 Estimates} finishes the proof of Theorem \ref{Thm:Main Theorem} using interpolation theory.

\begin{acknowledgements}
    The author would like to express his sincerest thanks to Andrey Gogolev for suggesting the problem, as well as for his patience and mentoring throughout the project.
\end{acknowledgements}

\section{Preliminaries}
\label{sec: Preliminaries}

\subsection{The Arzel\`a-Ascoli Argument}
\label{sec: The Arzela-Ascoli Argument}
We begin by recalling the following classical theorem:
\begin{theorem}[Arzel\`a-Ascoli]
    \label{thm: Arzela-Ascoli}
    Let $X$ be a compact Hausdorff space. Let $\mathcal{F}$ be a equicontinuous, pointwise bounded subset of $C^0(X)$. Then $\mathcal{F}$ is totally bounded in the uniform metric, and the closure of $\mathcal{F}$ is compact in $C(X)$.
\end{theorem}
We have the following important corollary:
\begin{cor}
    \label{cor: Compact Embedding}
    Let $M$ be a compact manifold and $0<r_1<r_2$. Then $C^{r_2}$ is compactly embedded in $C^{r_1}$. 
\end{cor}
In particular, our $C^2$ closed and bounded set $\mathcal{U}$ is $C^1$ compact. As a consequence, it will be sufficient to prove uniformity of all estimates within a $C^1$ neighborhood of a fixed $f\in\mathcal{U}$. By compactness, $\mathcal{U}$ is finitely covered by such neighborhoods and we can extract a uniform bound that works for every $f\in\mathcal{U}$. The cost is that we lose explicit information on the size of our constants when we apply this compactness argument.

\subsection{Anosov Diffeomorphisms}
\label{sec: Anosov Diffeomorphisms}
A $C^1$ diffeomorphism $f:M\rightarrow M$ of a compact manifold $M$ is said to be Anosov if there exists a continuous, nontrivial, $Df$-invariant splitting of the tangent bundle
\begin{equation}
    \label{eqn: Anosov Splitting}
    TM=E^u\oplus E^s
\end{equation}
and constants $0<\mu_-<\nu_-<1<\mu_+<\nu_+$, $C\geq 1$ such that
$$C^{-1}\mu_-^n||v^s||\leq ||D_xf^n(v^s)||\leq C\nu_-^n||v^s||,\hspace{1mm} \text{and}$$
\begin{equation}
\label{eqn: Definition of Expansion and Contraction rates}
    C^{-1}\mu_+^n||v^u||\leq ||D_xf^n(v^u)||\leq C\nu_+^n||v^u||,
\end{equation}
for every $x\in M$,$n\in\mathbb{N}$, $v^s\in E^s(x)$, and $v^u\in E^u(x)$.
The bundles $E^s$ and $E^u$ are called the stable and unstable subbundles, respectively.
In dimension two, we necessarily have $\dim(E^s)=\dim(E^u)=1$. We thus may identify $Df|_{E^s}$ and $Df|_{E^u}$ with the stable and unstable Jacobians which we denote by $D_sf$ and $D_uf$, respectively.

In order to obtain uniform estimates in Theorem \ref{Thm:Main Theorem}, it will be necessary to have uniformity of the expansion and contraction rates (\ref{eqn: Definition of Expansion and Contraction rates}). These rates will be bounded uniformly in a sufficiently small $C^1$-neighborhood of any fixed $f\in\mathcal{U}$. Thus, by the Arzela\`a-Ascoli argument, we have constants $0<\mu_0<\nu_0<1<\mu_1<\nu_1$, $C_{\mathcal{U}}\geq 1$ such that for every $f\in\mathcal{U}$, $x\in\mathbb{T}^2$, and $n\in\mathbb{N}$, we have 
\begin{equation}
\label{eqn:Anosov Condition}
    C_{\mathcal{U}}^{-1}\mu^n_0\leq|D_sf^n(x)|\leq C_{\mathcal{U}}\nu^n_0,\hspace{1mm}\text{and}\hspace{1mm}C^{-1}_{\mathcal{U}}\mu^n_1\leq|D_uf^n(x)|\leq C_{\mathcal{U}}\nu^n_1.
\end{equation}

The stable and unstable subbundles uniquely integrate to $f$-invariant foliations which we denote by $\mathcal{F}^{s,f}$ and $\mathcal{F}^{u,f}$, respectively. The leaves of these foliations can be characterized as follows:
\begin{equation}
    W^s_f(x)=\{y\in\mathbb{T}^2\hspace{1mm} |\hspace{1mm} d(f^n(x),f^n(y))\rightarrow 0, n\rightarrow\infty \},
\end{equation}

\begin{equation}
    W^u_f(x)=\{y\in\mathbb{T}^2\hspace{1mm} |\hspace{1mm} d(f^n(x),f^n(y))\rightarrow 0, n\rightarrow -\infty \}.
\end{equation}
The leaves of the stable and unstable foliations are as smooth as the diffeomorphism $f$, however the foliations themselves often have a lower degree of regularity; see section \ref{sec: Regularity of the Stable and Unstable Foliations} for a more detailed discussions of the regularity of the stable and unstable foliations. 

Since the stable and unstable manifolds through a point $x$ are (at least) $C^1$ immersed submanifolds of $\mathbb{T}^2$, we have an induced Riemannian metric on each leaf. For two points $x\in\mathbb{T}^2$ and $y\in W^s_f(x)$, we denote their distance within the leaf $W^s_f(x)$ by $d^s_f(x,y)$, and we similarly denote the metric on unstable leaves by $d^u_f$. The induced metrics allow us to define the local stable and unstable manifolds as
\begin{equation}
\label{eqn: Def of Local Stable Manifold}
    W^\sigma_{f,\delta}(x)=\{y\in W^\sigma_f(x)\hspace{1mm} | \hspace{1mm} d^\sigma_f(x,y)<\delta\},
\end{equation}
where $\delta>0$ and $\sigma=s,u$.

We say that a foliation of $\mathcal{F}$ a manifold $M$ is \emph{quasi-isometric} if there exist constants $a>0$ and $b\geq0$ such that for $x$ in the universal cover $\Tilde{M}$ and any $y\in \Tilde{W}(x)$, we have
$\Tilde{d}_{\Tilde{W}(x)}(x,y)\leq a\Tilde{d}(x,y)+b$, where $\Tilde{W}(x)$ denotes the lift of the leaf $W(x)$ to $\Tilde{M}$, $\Tilde{d}_{\Tilde{W}(x)}$ is the induced distance on $\Tilde{W}(x)$, and $\Tilde{d}$ is the lifted metric on $\Tilde{M}$. It is well known that the stable and unstable foliations of an Anosov diffeomorphism are quasi-isometric; see for instance the paper of Brin, Burago, and Ivanov \cite{Quasi-Isometry} for a proof which applies more generally to partially hyperbolic diffeomorphisms on $\mathbb{T}^3$.

The property of a diffeomorphism being Anosov is robust. That is, if $f$ is a $C^r$-Anosov diffeomorphism ($r\geq 1$) and $g$ is a $C^r$-diffeomorphism that is sufficiently $C^1$ close to $f$, then $g$ is also an Anosov diffeomorphism. In fact, Anosov diffeomorphisms enjoy a stronger property known as structural stability: $g$ will actually be topologically conjugate to $f$ by a conjugacy in the homotopy class of the identity. It is important in the present work that all of this can also be made quantitative:
\begin{theorem}[Strong Structural Stability]
    \label{thm: Strong Structural Stability}
    Let $f:M\rightarrow M$ be a $C^r$ Anosov diffeomorphism ($r\geq1$). For every $\delta>0$ there exists $\varepsilon>0$ such that if $g:M\rightarrow M$ is a $C^r$ diffeomorphism satisfying $d_{C^1}(f,g)<\varepsilon$, then $g$ is Anosov and there exists a homeomorphism $h:M\rightarrow M$ such that
    $d_{C^0}(h,\id)+d_{C_0}(h^{-1},\id)<\delta$. Moreover, $h$ is unique when $\delta$ is small enough.
\end{theorem}
See Katok and Hasselblatt \cite{katok_hasselblatt_1995} Theorem 18.2.1 for a proof. Together with the Arzel\`a-Ascoli argument described in the next section and Lemma \ref{lem:Uniform Holder}, this will account for much of the uniformity arguments in proving Theorem \ref{Thm:Main Theorem}.

For any Anosov diffeomorphism $f:\mathbb{T}^2\rightarrow\mathbb{T}^2$, there is a unique linear Anosov diffeomorphism $F_L$ in the homotopy class of $f$. By a theorem of Franks \cite{Franks69} and Mannings \cite{Manning74}, there exists a conjugacy $H$ homotopic to the identity conjugating $f$ and $F_L$: $H\circ f=F_L\circ H$. Moreover, the number of such conjugacies homotopic to the identity is finite and is determined by the number of fixed points of $F_L$. In particular, any two Anosov diffeomorphisms in the homotopy class of a fixed linear Anosov diffeomorphism are conjugate, and moreover, there are only finitely many such conjugacies within the homotopy class of the identity. We will always assume that the conjugacy $h$ between two $f,g\in\mathcal{U}$ is homotopic to the identity.

\subsection{Regularity of the Stable and Unstable Foliations}
\label{sec: Regularity of the Stable and Unstable Foliations}

As previously mentioned, the leaves of the stable and unstable foliations of an Anosov diffeomorphism $f$ will be as smooth as $f$ itself; in our setting, this means that the leaves $W^s_f(x)$ and $W^u_f(x)$ will be $C^2$ submanifolds of $\mathbb{T}^2$ for every $x$. The regularity of the foliations themselves is a much more delicate matter. It is well known that there exists $0<\alpha<1$ (which can be made explicit in terms of the rates of expansion and contraction) such that the stable and unstable foliations are $\alpha$-H\"older continuous. In general, however, it is rare for the regularity to exceed this except in special circumstances.

In low dimensions, and in particular with the area-preserving hypothesis, we in fact have higher regularity of the foliations. For a general Anosov diffeomorphism of $\mathbb{T}^2$, it follows from Hasselblatt \cite{Hasselblatt_1994} that the stable and unstable foliations are both $C^{1+\varepsilon}$ for some $\varepsilon>0$.

When $f$ is area-preserving we can say even more. A function $\psi$ is said to belong to the Zygmund class (written $\psi\in C^{Zyg}$) if it has modulus of continuity $O(x\log|x|)$. This in particular it implies that $\psi$ is $\alpha$-H\"older continuous for every $\alpha<1$, but it is a weaker condition than being Lipschitz continuous. When $f:\mathbb{T}^2\rightarrow\mathbb{T}^2$ is a $C^k$ ($k\geq 2$) area-preserving Anosov diffeomorphism, then it follows from a result of Katok and Hurder \cite{Katok_Hurder} that stable and unstable foliations are $C^{1+Zyg}$ (see also \cite{Foulon_Hasselblatt} for the corresponding result for Anosov flows on $\mathbb{T}^3$).

The importance of the regularity of the stable and unstable foliations enters our construction in Section \ref{sec: Proof of the Main Theorem: The Construction} through the holonomy maps.
Let $b\in W^s_f(a)$. Then the stable holonomy map of $f$ from $a$ to $b$ is the map $\hol^{s,f}_{a,b}:W^u_{f,loc}(a)\rightarrow W^u_{f,loc}(b)$ such that $\hol^{s,f}_{a,b}(a)=b$ obtained by sliding points of $W^u_{f,loc}$ along stable leaves until they intersect $W^u_{f,loc}(b)$. When the stable and unstable foliations have global product structure (which in particular is always the case for Anosov diffeomorphisms on tori), then the stable holonomy can be continuously extended to a map defined on the entire unstable manifold of $f$: $\hol^{s,f}_{a,b}:W^u_{f}(a)\rightarrow W^u_{f}(b)$. We define the unstable holonomy between stable manifolds analogously.

The stable and unstable holonomies play a crucial role in the construction of the conjugacy $\overline{h}_N$ and it is therefore important that these holonomies be at least $C^{1+\varepsilon}$. The holonomy maps are as regular as the foliations themselves, and thus are $C^{1+\varepsilon}$, and $C^{1+Zyg}$ for area-preserving diffeomorphisms. Two points deserve emphasis here. First, if we wish to show that $\overline{h}_N$ is $C^k$ for $k\geq 2$ and get corresponding estimates on $d_{C^k}(f,\overline{f}_N)$, we must have that the stable and unstable foliations of both $f$ and $g$ are $C^k$. Unfortunately, this only happens in the rare circumstance when both $f$ and $g$ are smoothly conjugated to their linear parts, and hence smoothly conjugated to each other. In general, it seems a different construction entirely would be necessary to establish higher regularity of $\overline{h}_N$. Second, our argument that both foliations are $C^{1+\varepsilon}$ heavily relies on the fact that both stable and unstable distributions are one-dimensional. To generalize the construction to Anosov diffeomorphisms on higher dimensional tori, we would need additional bunching assumptions on the expansion and contraction rates to ensure enough regularity of the holonomy maps. However, even under such assumptions, it is not immediately clear how to generalize our techniques to higher dimensions.

To prove that $\overline{h}_N$ is $C^{1+\alpha}$ in Theorem \ref{Thm:Main Theorem}, we will first prove regularity along the stable and unstable foliations separately with uniformity. We say a continuous function $\psi$ belongs to $C^{1+\alpha}_s$ if $\psi$ is uniformly $C^{1+\alpha}$ when restricted to any stable leaf; that is, there exists some constant $C>0$ such that for every stable leaf $W^s_f$, we have
$|D_s\psi|_{W^s_f}|\leq C$. We similarly define $C^{1+\alpha}_u$. The following result of Journ\'e \cite{Journé1988} tells us that to verify smoothness of $\overline{h}_N$, it is enough to establish uniform regularity on each foliation:
\begin{theorem}[Journ\'e's Lemma]
    \label{thm: Journe Lemma}
    $C^{1+\alpha}=C^{1+\alpha}_s\cap C^{1+\alpha}_u$ for every $0<\alpha<1$.
\end{theorem}

\subsection{SRB Measure and Area}
\label{sec: SRB Measure and area}
In the present paper we will be primarily concerned with area-preserving Anosov diffeomorphisms. However, as mentioned in Remark \ref{rem: area-Preserving Hypothesis}, this hypothesis will not be needed for the construction of the conjugacy $\overline{h}_N$. We will therefore briefly review the definition and main properties of SRB measures that will be needed for the construction in Section \ref{sec: Proof of the Main Theorem: The Construction}. For a more detailed discussion of SRB measures, see the survey of Young \cite{Young}.

Intuitively, the SRB measure of a diffeomorphism $f$ is the invariant measure most compatible with area when area is not preserved. For transitive $C^{1+\alpha}$ Anosov diffeomorphisms, there are several definitions of SRB measures, and it is a highly non-trivial result that they are all equivalent. For our purposes we will define the SRB measure in terms of its conditional measures along the unstable foliation (we will later on use the characterization of the SRB measure as the equilibrium state corresponding to the geometric potential, see Section \ref{sec: Proof of EE}). To make this precise, we begin by recalling the definition of a partition subordinate to the unstable foliation.
\begin{definition}
    \label{def: Partition subordinate to unstable}
    Given an Anosov diffeomorphism $f:M\rightarrow M$ and a measure $\mu$, a measurable partition $\xi$ is said to be subordinate to the unstable foliation $\mathcal{F}^{u,f}$ whenever for $\mu$ a.e. $x$ we have
    \begin{enumerate}
        \item $\xi(x)\subset W^u_f(x)$
        \item $\xi(x)$ contains an open subset of a neighborhood of $x$ in $W^u_f(x)$ in the submanifold topology.
    \end{enumerate}
\end{definition}

\begin{definition}
    Let $f:M\rightarrow M$ be a $C^2$ Anosov diffeomorphism. An $f$-invariant measure $\mu_f$ is said to be an SRB measure if for every measurable partition $\xi$ subordinate to $\mathcal{F}^{u,f}$, the conditional measure of $\mu_f$ on the partition element $\xi(x)$, denoted by $\mu_f^{\xi(x)}$, is absolutely continuous with respect to the Lebesgue measure induced by the Riemannian metric on $W^u_f(x)$.
\end{definition}
By taking a sequence of increasing and subordinate partitions to $\mathcal{F}^{u,f}$, we can define a ``maximal" conditional measure $\mu^u_{f,x}$ on $W^u_f(x)$ (See da la Llave \cite{delaLlave} Section 3 for details). We will from now on work only with these ``maximal" conditional measures.

The densities of the conditional measures $\mu^u_{f,x}$ are unique up to scalar multiples. If $x_0\in W^u_f(x)$ then the we denote by $\omega^u_f(x,x_0)$ the density of $\mu^u_{f,x_0}$ with respect to Lebesgue such that $\omega^u_f(x_0,x_0)=1$. Then $\omega^u_f(x,x_0)$ is $C^1$ along $W^u_f(x_0)$ and is given by the formula
\begin{equation}
    \label{eqn: Conditional Density Normalized to 1 at a point Formula}
    \omega^u_f(x,x_0)=\prod_{i=1}^\infty \frac{D_uf(f^{-n}(x))}{D_uf(f^{-n}(x_0))}.
\end{equation}

The area-preserving hypothesis means that each $f\in\mathcal{U}$, the SRB measure $\mu_f$ of $f$ is in the smooth measure class of the Lebesgue measure on $\mathbb{T}^2$; that is $d\mu_f=\varphi_fdm$, where $\varphi_f\in C^1$ is strictly positive (Note however that $f$ may not preserve the Lebesgue measure itself). Then there exists $c_f\geq 1$ such that $c_f^{-1}\leq\varphi_f\leq c_f$. We will need in the proof of Theorem \ref{Thm:Main Theorem} uniformity of $C_f$.
\begin{lemma}
\label{lem:Uniform bounds on area Density}
    There exists a constant $c_{\mathcal{U}}>1$ such that for every $f\in\mathcal{U}$, its invariant density $\varphi_f$ satisfies
    $c_{\mathcal{U}}^{-1}<\varphi_f<c_{\mathcal{U}}$.
\end{lemma}
To prove this lemma, we will need the following uniform version of local product structure for Anosov diffeomorphisms.
\begin{theorem}
    \label{thm: Local Product Structure Uniform}
    Let $f:M\rightarrow M$ be an Anosov diffeomorphism. There is a $C^1$ neighborhood $\mathcal{N}$ of $f$ and $\varepsilon_0>0$ such that the following holds: for every $0<\varepsilon<\varepsilon_0$ there exists a $\delta>0$ such that if $g\in\mathcal{N}$ and $x,y\in M$ are such that $d(x,y)<\delta$, then $W^s_{g,\varepsilon}(x)$ and $W^u_{g,\varepsilon}(y)$ intersect transversely.
\end{theorem}
For a proof of local product structure including the uniformity statement, see Cooper \cite{cooper2021structural} Theorem 4.9. This result also follows from the standard local product structure of a fixed $f$, strong structural stability of $f$, and Theorem \ref{lem:Uniform Holder}.

\begin{proof}[Proof of Lemma \ref{lem:Uniform bounds on area Density}]
    Let $Jf$ denote the Jacobian of $f$ with respect to the standard metric on $\mathbb{T}^2$. The condition that $f$ is area-preserving is equivalent to $Jf$ being a coboundary over $f$; that is, there exists a continuous function $u_f:\mathbb{T}^2\rightarrow\mathbb{R}$ such that $Jf=u_f\circ f-u_f$. It can be shown in fact that $u_f=\log(\varphi_f)$. 

    Now let $x\in\mathbb{T}^2$ be arbitrary and $y\in W^u_f(x)$. Then using the cohomological equation we have
    $$
    \log(\varphi_f(x))-\log(\varphi_f(y))=\sum_{n=0}^{\infty}\left(Jf(f^n(x))-Jf(f^n(y))\right).
    $$
    Therefore 
    $$
    |\log(\varphi_f(x))-\log(\varphi_f(y))|\leq \sum_{n=0}^{\infty}\left|Jf(f^n(x))-Jf(f^n(y))\right|\leq
    $$
    $$
    |Jf|_{Lip}\sum_{n=0}^\infty d(f^n(x),f^n(y))\leq 
    |Jf|_{Lip}d(x,y)\sum_{n=0}^\infty \nu_1^n< Cd(x,y),
    $$
    where $C>0$ is uniform. Thus
    $$
    \varphi_f(x)\leq \varphi_f(y)e^{Cd(x,y)}
    $$
    whenever $y\in W^u_f(x)$. By replacing $f$ by $f^{-1}$ and noting that $\varphi_f$ is still the invariant density for $f^{-1}$, we conclude the same is true whenever $y\in W^s_f(x)$. Let $0<\varepsilon<\varepsilon_0$ and $\delta>0$ be as in Theorem \ref{thm: Local Product Structure Uniform}. Then if $d(x,y)<\delta$, $W^s_{f,\varepsilon}(x)\cap W^u_{f,\varepsilon}(y)=:{z}$, and we have 
    $$
    \varphi_f(x)<\varphi_f(z)e^{Cd(x,z)}<\varphi_f(y)e^{C(d(x,z)+d(z,y))}<\varphi_f(y)e^{2C\varepsilon}.
    $$
    By compactness, we may cover $\mathbb{T}^2$ by finitely $\delta$-balls. Then any two points $x,y\in\mathbb{T}^2$ can be connected by at most $M$ points $x=x_0, x_1,\cdots, x_k=y$, $k\leq M$, with $d(x_i,x_{i-1})<\delta$ for $i=1,\cdots k$, and moreover $M$ depends only on the choice of covering which is independent of $f$. Therefore,
    $$
    \varphi_f(x)<\varphi_f(x_1)e^{2C\varepsilon}<\cdots<\varphi_f(y)e^{2CM\varepsilon}.
    $$
    Since  $\varphi_f$ is a continuous density for a probability measure, there exists at least one $y\in\mathbb{T}^2$ such that $\varphi_f(y)=1$. Therefore, for every $x\in\mathbb{T}^2$, we have 
    $$
    \varphi_f(x)<e^{2CM\varepsilon},
    $$
    and the lower bound follows similarly.
\end{proof}

Finally, the SRB measure has an important property known as local product structure. To define it, we first recall the definition of a hyperbolic rectangle:
\begin{definition}
    \label{def: Rectangle}
    Let $\delta>0$ be as in Theorem \ref{thm: Local Product Structure Uniform}. A set $R\subset\mathbb{T}^2$ is called a rectangle if $\diam(R)\leq \delta$ and is closed under the local product structure bracket operation: if $x,y\in R$, then
    $[x,y]:=W^s_f(x)\cap W^u_f(y)$ exists and is contained in $R$. 
\end{definition}

\begin{definition}
    \label{def: Local Product Structure of Measure}
    A measure $\mu$ is said to have local product structure with respect to the stable and unstable foliations if for every $x\in\mathbb{T}^2$ and every rectangle $R$ containing $x$, there exists measures $\mu^u$ and $\mu^s$ such that
    $\mu|_R<<\mu^u\otimes\mu^s$.
\end{definition}
It is an important property that the SRB measure $\mu_f$ has local product structure. In fact, more can be said about the density $\frac{d\mu_f}{d(\mu^u\otimes\mu^s)}$. By Rokhlin's disintegration theorem (\cite{Rokhlin_1967}; also see Einsiedler and Ward \cite{einsiedler2010ergodic} Chapter 5 for a more recent exposition on conditional measures and disintegrations), we have 
$$
\int_R\psi d\mu_f=\int_{R\cap W^s_f(x)} \int_{R\cap W^u_f(y)} \psi(y,z) d\mu^u_{f,y}(z)d\hat{\mu}_f(y),
$$
where $\mu^u_{f,y}$ is the conditional measure of $\mu_f$ on $W^u_f(y)$, and $\hat{\mu}_f$ is the quotient measure defined by 
$\hat{\mu}_f(A):=\mu_f(\pi^{-1}(A))$ where $\pi:R\rightarrow W^s_f(x)$ is the projection onto the transversal. The holonomy map $\hol^{s,f}_{x,y}:W^u_f(x)\rightarrow W^u_f(y)$ is absolutely continuous with respect the conditional measures, and so we can write
$$
\int_{R\cap W^s_f(x)} \int_{R\cap W^u_f(y)} \psi(y,z) d\mu^u_{f,y}(z)d\hat{\mu}_f(y)=$$
$$
\int_{R\cap W^s_f(x)} \int_{R\cap W^u_f(x)} \left(\psi\frac{d(\hol^{s,f}_{x,y})_*\mu^u_{f,y}}{d\mu^u_{f,x}}\right)\circ\hol^{s,f}_{x,y}(y,z) d\mu^u_{f,x}(z)d\hat{\mu}_f(y).
$$
We have the following expression for the Radon-Nikodym derivative (see Barreira and Pesin \cite{barreira_Pesin} Theorem 4.4.1):
\begin{equation}
    \frac{d(\hol^{s,f}_{x,y})_*\mu^u_{f,y}}{d\mu^u_{f,x}}(z)=
    \prod_{k=0}^\infty\frac{D_uf^{-1}(f^k(\hol^{s,f}_{x,y}(z)))}{D_uf^{-1}(f^k(z))}.
\end{equation}
From this formula, it is easy to see that $\Psi_{f,x}(y,z):=\frac{d(\hol^{s,f}_{x,y})_*\mu^u_{f,y}}{d\mu^u_{f,x}}(z)$ is Lipschitz continuous in both $y$ and $z$. Moreover, the Lipschitz seminorm of $\Psi_{f,x}$ is uniformly bounded in both $x\in\mathbb{T}^2$ and $f\in\mathcal{U}$.

\section{Proof of Main Theorem: The Construction}
\label{sec: Proof of the Main Theorem: The Construction}
\subsection{Uniformity of the Initial Conjugacy}
\label{sec: Uniformity of the Initial Conjugacy}
Before we can begin constructing the new conjugacy $\overline{h}_N$, we first must investigate the regularity of our initial conjugacy $h$. It is well known that any topological conjugacy between two Anosov diffeomorphisms is automatically H\"older continuous; that is, there exists $\alpha>0$ such that for all $x,y\in\mathbb{T}^2$, $d(h(x),h(y))\leq |h|_\alpha d(x,y)^\alpha$. See for instance Katok and Hasselblatt \cite{katok_hasselblatt_1995} Theorem 19.1.2 for a proof of this important property in a more general setting. At several steps the H\"older exponent $\alpha$ of $h$ and the associated seminorm $|h|_\alpha$ will come up when establishing the uniformity claims of Theorem \ref{Thm:Main Theorem}.
However, given any two conjugate Anosov diffeomorphisms $f,g\in\mathcal{U}$, $hf=gh$, it is not clear a priori that they are all H\"older continuous with the same H\"older exponent. The goal of this subsection is to prove the following lemma:
\begin{lemma}
\label{lem:Uniform Holder}
    There exists $\alpha_0>0$ and $C_{\alpha_0}>0$ such that for any $f,g\in \mathcal{U}$ and any conjugacy $h$ between $f$ and $g$, $h$ is $\alpha_0$-H\"older continuous with $C_{\alpha_0}^{-1}\leq |h|_{\alpha_0}\leq C_{\alpha_0}$.
\end{lemma}
Actually, we will only need uniformity of $h$ restricted to stable and unstable manifolds, from which Lemma \ref{lem:Uniform Holder} follows. To this end, we will prove uniformity along unstable manifolds. The case of stable manifolds is completely symmetric. We will prove the following:

\begin{lemma}
\label{lem: Non local Lipschitz along unstables}
   There exists $\delta>0$ and $C_\delta>0$ such that $d^u_g(h(x),h(y))\leq C_\delta d^u_f(x,y)$ whenever $d^u_f(x,y)\geq\delta$.
\end{lemma}
Let us see how Lemma \ref{lem:Uniform Holder} follows from this:
\begin{proof}[Proof of Lemma \ref{lem:Uniform Holder}]
    It suffices to prove that $h$ is locally H\"older. Let $y\in W^u_f(x)$ be such that $d^u_f(x,y)\leq\delta$. Let $n\in\mathbb{N}$ be minimal such that $\delta<d^u_f(f^n(x),f^n(y))\leq\nu_1\delta$. Then since $h=g^{-n}hf^n$, we have 
    $$
    d^u_g(h(x),h(y))=d^u_g(g^{-n}hf^n(x),g^{-n}hf^n(y))\leq 
    \mu_0^{-n}d^u_g(h(f^n(x)),h(f^n(y)))\leq C_\delta\mu_0^{-n} d^u_f(f^n(x),f^n(y)).
    $$
    Let $\alpha_0>0$ be such that $\mu_0^{-1}\nu_1^{\alpha_0}=1$. Since $\delta<d^u_f(f^n(x),f^n(y))\leq\nu_1\delta$, we can find $C_{\alpha_0}>0$ such that $d^u_f(f^n(x),f^n(y))\leq C_{\alpha_0}d^u_f(f^n(x),f^n(y))^{\alpha_0}$. Then
    $$
    C_\delta\mu_0^{-n} d^u_f(f^n(x),f^n(y))<
    C_\delta C_{\alpha_0}\mu_0^{-n} d^u_f(f^n(x),f^n(y))^{\alpha_0}
    \leq$$
    $$
    C_\delta C_{\alpha_0}\mu_0^{-n}\nu_1^{\alpha_0n}d^u_f(x,y)^{\alpha_0}=C_\delta C_{\alpha_0}d^u_f(x,y)^{\alpha_0}.    
    $$
    Absorbing $C_\delta$ into $C_{\alpha_0}$, we are done.
\end{proof}

We prove Lemma \ref{lem: Non local Lipschitz along unstables} using the following two lemmas, the first of which says that the unstable foliations are uniformly quasi-isometric:
\begin{lemma}
    \label{lem:Universal Quasi-Isometric}
    $\exists a>0,b\geq0$ such that $\forall f\in\mathcal{U}$
    $\Tilde{d}^u_f(x,y)\leq a\Tilde{d}(x,y)+b$ for every $y\in \Tilde{W}^u_f(x)$.
\end{lemma}
Recall that $\Tilde{W}^u_f$ denotes the lift of the unstable leaf to the universal cover, $\Tilde{d}^u_f$ denotes the leafwise distance on the lifted leaf, and $\Tilde{d}$ denotes the Riemannian distance on the universal cover. 
\begin{proof}
    Fix $f\in\mathcal{U}$. Since the unstable foliation is quasi-isometric, there exists $a_f,b_f>0$ such that 
    $$
    d^u_f(x,y)\leq a_fd(x,y)+b_f,
    $$
    for $y\in W^u_f(x)$. The goal is to show that for $g$ sufficiently $C^1$ close to $f$ we can express $\Tilde{W}^u_g(a)$ as a Lipschitz graph over $\Tilde{W}^u_f(b)$ for some $b$. To make this precise we assume that $d_{C^1}(f,g)<\delta$ is sufficiently small so that there exists $\theta<\frac{\pi}{2}$ such that $\measuredangle(E^u_f,E^u_g)\leq \theta$. By strong structural stability, given any $\epsilon>0$, after possibly decreasing $\delta$, $f$ is conjugate to $g$ by a homeomorphism $h_g$ that satisfies $d_{C^0}(h_g,\id)+d_{C^0}(h_g^{-1},\id)<\epsilon.$
    Given a point $a\in\mathbb{T}^2$, let $b=h_g(a)$. Then $\Tilde{W}^u_g(a)$ lies in an $\epsilon$-neighborhood of $\Tilde{W}^u_f(b)$, and every point in $\Tilde{W}^u_g(a)$ can be connected to $\Tilde{W}^u_f(b)$ by a stable manifold of $f$ of length less than $C\epsilon$, where $C>0$ is some constant depending only on $f$. Define $\varphi_g:\Tilde{W}^u_f(b)\rightarrow \Tilde{W}^u_g(a)$ by 
    $\varphi_g(x)=\Tilde{W}^s_f(x)\cap\Tilde{W}^u_g(b)$. Then $\varphi_g$ is a Lipschitz function with Lipschitz constant depending only on $f$ and the maximum angle $\theta$ between $E^u_f$ and $E^u_g$, but not depending on the point $a$.
    Therefore, for any $x,y\in\Tilde{W}^u_g(a)$,
    $$
    \Tilde{d}^u_g(x,y)=\Tilde{d}^u_g(\varphi_g(\overline{x}),\varphi_g(\overline{y}))\leq |\varphi_g|_{Lip}\Tilde{d}^u_f(\overline{x},\overline{y})\leq
    a_f|\varphi_g|_{Lip}\Tilde{d}(\overline{x},\overline{y})+b_f|\varphi_g|_{Lip}\leq
    $$
    $$
    a_f|\varphi_g|_{Lip}\left(\Tilde{d}(x,y)+\Tilde{d}(x,\overline{x})+\Tilde{d}(y,\overline{y}) \right)+b_f|\varphi_g|_{Lip}\leq
    a_f|\varphi_g|_{Lip}\Tilde{d}(x,y)+b_f|\varphi_g|_{Lip}+2C\epsilon a_f|\varphi_g|_{Lip}.
    $$
This gives us uniformity in a $C^1$ neighborhood of $f$. By the Arzel\`a-Ascoli compactness argument, we get uniformity for every $f\in\mathcal{U}$.    
\end{proof}
\begin{rem}
  When $f$ is sufficiently close to linear we may take $b=0$ (See for instant the paper of Gogolev and Guysinsky \cite{Gogolev_Guysinsky}). See also the proof of Lemma \ref{lem: Uniform change in leaf metric}  
\end{rem}

\begin{lemma}
    \label{lem: Non local Lipschitz on universal cover}
     $\forall \delta>0, \exists C_\delta>0$ such that $\Tilde{d}(h(x),h(y))\leq C_\delta \Tilde{d}(x,y)$ whenever $\Tilde{d}(x,y)\geq\delta$.
\end{lemma}
\begin{proof}
    Fix $f,g\in\mathcal{U}$ and let $h$ be the unique homeomorphism in the homotopy class of the identity conjugating $f$ and $g$ ($hf=gh$). By structural stability, if we take $\overline{f}$ and $\overline{g}$ sufficiently $C^1$ close to $f$ and $g$, respectively, then $f$ is conjugate to $\overline{f}$ by a conjugacy $h_f$ that is $C^0$ close to the identity, and in the same homotopy class as the identity, and similarly for $g$ and $\overline{g}$. Therefore $\overline{f}$ and $\overline{g}$ are conjugate by a homeomorphism $\overline{h}$ in the same homotopy class as the identity and is $C^0$ close to $h$. 

    Since $\overline{h}$ is homotopic to the identity, its lift to the universal cover satisfies $\overline{h}(x+\overline{n})=\overline{h}(x)+\overline{n}$ for every $\overline{n}\in\mathbb{Z}^2$. Now let $\Tilde{d}(x,y)\geq\delta$, and $x=\overline{x}+\overline{n}$, $y=\overline{y}+\overline{m}$, where $\overline{x},\overline{y}\in [0,1]^2$ and $\overline{n},\overline{m}\in\mathbb{Z}^2$. Then we have
    $$
    |\overline{h}(x)-\overline{h}(y)|=|\overline{h}(\overline{x}+\overline{n})-\overline{h}(\overline{y}+\overline{m})|\leq |\overline{h}(\overline{x})-\overline{h}(\overline{y})|+|\overline{n}-\overline{m}|\leq 2\diam(\overline{h}([0,1]^2))+|\overline{n}-\overline{m}|.
    $$
    Since $d_{C^0}(h,\overline{h})<\varepsilon$, we have 
    $\diam(\overline{h}([0,1]^2))<\diam(h([0,1]^2))+2\varepsilon$. Then, since $\Tilde{d}(x,y)\geq\delta$, we have
    $$
    |\overline{h}(x)-\overline{h}(y)|\leq C_\delta \Tilde{d}(x,y),
    $$
    where $C_\delta>0$ is uniform in neighborhoods of $f$ and $g$. Finally, by the Arzel\`a-Ascoli argument, $C_\delta>0$ can be chosen uniform for all $f,g\in\mathcal{U}$, depending only on $\delta>0$ and $\mathcal{U}$.
\end{proof}
The proof of Lemma \ref{lem: Non local Lipschitz along unstables} now follows easily:
\begin{proof}[Proof of Lemma \ref{lem: Non local Lipschitz along unstables}]

    When $\delta>b$, $d^u_f(x,y)\geq\delta$ implies that $\Tilde{d}(x,y)\geq\frac{\delta-b}{a}=\delta'>0$ by Lemma \ref{lem:Universal Quasi-Isometric}. Therefore we may assume that $\Tilde{d}(x,y)\geq\delta'$.
    By Lemmas \ref{lem:Universal Quasi-Isometric} and \ref{lem: Non local Lipschitz on universal cover} we have
    $$
    d^u_g(h(x),h(y))\leq a\Tilde{d}(h(x),h(y))+b\leq
    aC_{\delta'}\Tilde{d}(x,y)+b.
    $$
    Since $\Tilde{d}(x,y)\geq\delta'$, we can find $C'>0$ depending only on $a,b,$ and $C_{\delta'}$ such that 
    \begin{equation*}
        aC_{\delta'}\Tilde{d}(x,y)+b\leq C'\Tilde{d}(x,y)\leq C'd^u_f(x,y).
    \end{equation*}
\end{proof}

\subsection{The Construction}
\label{sec: The Construction}
The main difficulty in the construction is ensuring that it is done uniformly across $f$ and $g$ in our set $\mathcal{U}$; that is, keeping track of how our choices in the construction enter into our constants in Theorem \ref{Thm:Main Theorem}. Throughout the construction we shall state lemmas we will need for the estimates, but we will defer the proof of the most technical of these lemmas to the next subsection.

We begin by recalling the main idea of \cite{delaLlave} in more detail. If the unstable periodic data between $C^r$ Anosov diffeomorphisms $f$ and $g$ are matched by a conjugacy $h$, then since the SRB measure $\mu_f$ is the equilibrium state of the potential $\psi_f=-\log(D_uf)$, we have the $h$ pushes the SRB measure of $f$ to the SRB measure of $g$: $h_*\mu_f=\mu_g$. It then follows that $h$ pushes forward the conditional measures of $\mu_f$ along unstable leaves to the corresponding conditional measures of $\mu_g$. To make this precise, if we fix $x_0\in\mathbb{T}^2$, and define
$$I^u_{f,x_0}(x)=\int_{x_0}^x d\mu^u_{f,x_0},$$ 
where the integral is taken over the unstable manifold of $x_0$ from $x_0$ to $x$, and similarly define $I^u_{g,h(x_0)}(x)$, then the agreement of the the unstable periodic data implies that $I^u_{f,x_0}(x)=I^u_{g,h(x_0)}(h(x))$. We can then express $h$ as the composition of $C^r$ functions $h=(I^u_{g,h(x_0)})^{-1}\circ I^u_{f,x_0}$, which shows that $h$ is $C^r$ when restricted to each (dense) unstable leaf. From here de la Llave is able to conclude that $h\in C^r_u$, and a symmetric argument using the stable periodic data and the SRB measures for $f^{-1}$ and $g^{-1}$ shows that $h\in C^r_s$. The proof is then finished by an application of Journ\'e's lemma (Theorem \ref{thm: Journe Lemma}).

Letting $x_0\in\mathbb{T}^2$ be a fixed point of $f$, this motivates us to try and define $h_N$ on $W^u(x_0)$ by the formula 
$(I^u_{g,h(x_0)})^{-1}\circ I^u_{f,x_0}(x)$.
By minimality of the unstable foliation, this defines $h_N$ on a dense subset of $\mathbb{T}^2$, but unfortunately, this function cannot in general be extended to even a continuous function on the whole manifold. This is because two points $a,b\in W^u(x_0)$ may be close in the metric on $\mathbb{T}^2$, but far apart in the leaf, and so there is no reason that $h_N(a)$ should be close to $h_N(b)$. To circumvent this issue, we will instead start by defining $h_N$ on a compact segment of unstable manifold. Let $x_1\in W^u(x_0)\cap W^s(x_0)$, and denote the unstable segment between $x_0$ and $x_1$ by $A^u_f$ and the stable segment between $x_0$ and $x_1$ by $A^s_f$. We will also write $A^u_g:=h(A^u_f)$ and $A^s_g:=h(A^s_f)$. For the sake of our construction, the points $x_0$ and $x_1$ may be arbitrary heteroclinic points. For $\Tilde{f}\in\mathcal{U}$ that is $C^1$-close to $f$, we will choose the corresponding heteroclinic points $\Tilde{x}_0$ and $\Tilde{x}_1$ according to the structural stability conjugacy; see Section \ref{sec: Uniformity of the New Conjugacy} for details. 

\begin{figure}
    \centering
    \includegraphics[scale=0.25]{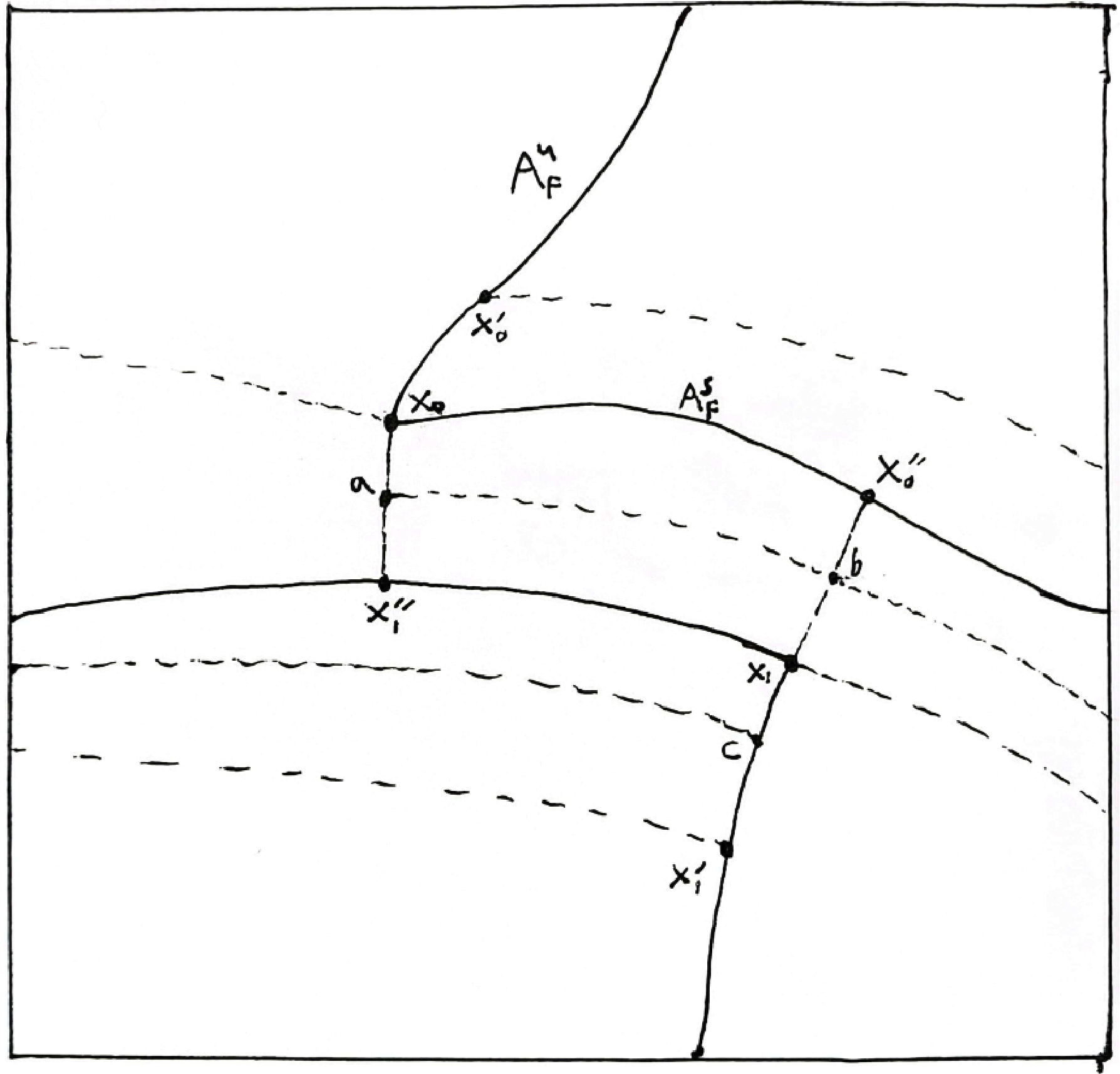}
    \caption{Configuration of $A^u_f$ and $A^s_f$}
    \label{fig:Figure1}
\end{figure}
Extending the stable manifold of the segment $A^s_f$ until the first intersections with $A^u_f$, we obtain two points $x_0',x_1'\in A^u_f$. The typical configuration of such points is shown in Figure \ref{fig:Figure1}. Note that we can extend $A^u_f$ to obtain two points $x_0'',x_1''\in A^s_f$. This lets us divide $A^u_f$ into three subintervals, $A^u_{f,1}:=[x_0,x_0']_u$, $A^u_{f,2}:=[x_0',x_1']_u$, and $A^u_{f,3}:=[x_1',x_1]_u$. Likewise, set $A^u_{g,i}:=h(A^u_{f,i})$, for $i=1,2,3$.

We normalize $\mu^u_{f,x_0}$ and $\mu^u_{g,h(x_0)}$ so that 
$\mu^u_{f,x_0}(A^u_{f,1})=\mu^u_{g,h(x_0)}(A^u_{g,1})=1$. We then define 
$h_N|_{A^u_{f,1}}:A^u_{f,1}\rightarrow A^u_{g,1}$ by
\begin{equation}
\label{def:h_N|A^u_{f,1}}
h_N|_{A^u_{f,1}}(x)= (I^u_{g,h(x_0)})^{-1}\circ I^u_{f,x_0}(x).
\end{equation}
Similarly, we normalize the measures $\mu^u_{f,x_0}$ and $\mu^u_{g,h(x_0)}$ so that 
$\mu^u_{f,x_0'}(A^u_{f,2})=\mu^u_{g,h(x_0')}(A^u_{g,2})=1$ and define $h_N|_{A^u_{f,2}}(x)$ by the same expression, and likewise for defining $h_N|_{A^u_{f,3}}(x)$. Observe that 
$h_N|_{A^u_{f,1}}(x_0)=h(x_0)$, 
$h_N|_{A^u_{f,1}}(x_0')=h_N|_{A^u_{f,2}}(x_0')=h(x_0')$, 
$h_N|_{A^u_{f,2}}(x_1')=h_N|_{A^u_{f,3}}(x_1')=h(x_1')$, and finally
$h_N|_{A^u_{f,3}}(x_1)=h(x_1)$. Therefore $h_N|_{A^u_f}$ is well-defined by these expressions. Moreover, by construction $h_N(x_0)$, $h_N(x_0')$, $h_N(x_1')$, and $h_N(x_1)$ all lie on the same stable manifold of $g$. In general, however, $h_N$ will not map stable manifolds of $f$ to stable manifolds of $g$.

Defined this way, $h_N|_{A^u_f}$ is piecewise $C^2$, with possible discontinuities in the derivative at the points $x_0'$ and $x_1'$. To obtain differentiability at these points, we will reparameterize the intervals $A^u_{f,i}$, $i=1,2,3$. To make this precise, let $\sigma_i:A^u_{f,i}\rightarrow A^u_{f,i}$ be an orientation-preserving $C^2$ diffeomorphism and consider 
$h_{N,\sigma_i}:=h_N|_{A^u_{f,i}}\circ\sigma_i :A^u_{f,i}\rightarrow A^u_{g,i}$, $i=1,2,3$. Then
$$
(h_{N,\sigma_i})'(x)=(h_N|_{A^u_{f,i}})'(\sigma_i(x))\sigma_i'(x).
$$
We may therefore choose the reparameterizations such that 
$(h_N|_{A^u_{f,1}})'(\sigma_1(x_0'))\sigma_1'(x_0')=
(h_N|_{A^u_{f,2}})'(\sigma_2(x_0'))\sigma_2'(x_0')$ and
$(h_N|_{A^u_{f,2}})'(\sigma_2(x_1'))\sigma_2'(x_1')=
(h_N|_{A^u_{f,3}})'(\sigma_3(x_1'))\sigma_3'(x_1')$. The following lemma says that we can do so without sacrificing our estimates in Section \ref{sec: Proof the Main Theorem the C^0 Estimates}.
\begin{lemma}
\label{lem:reparam}
    We can choose the $\sigma_i$ to be $C^0$-arbitrarily close to the identity with arbitrary values of the the first and second derivatives of $\sigma_i$ at the endpoints. 
\end{lemma}

By computing the second derivatives of the $h_{N,\sigma_i}$, we can see, using Lemma \ref{lem:reparam}, that the $\sigma_i$ can be chosen so that the first and second derivatives of the $h_{N,\sigma_i}$ at overlapping endpoints agree. Therefore we can combine these definitions and obtain a single $C^2$ function $h_{N,\sigma}|_{A^u_f}$. Notice that we are also able to use Lemma \ref{lem:reparam} to choose the value of the derivative of $h_{N,\sigma}|_{A^u_f}$ at the points $x_0,x_0',x_1',$ and $x_1$. 
Suppose for the the moment that $h\in C^1$ and consider a pair of points $a\in\mathbb{T}^2$ and $b\in W^u_f(a)\cap W^s_f(a)$. Then we have the relation
$$
h(x)=\hol^{s,g}_{h(a),h(b)}\circ h\circ\hol^{s,f}_{b,a}(x),
$$
for every $x\in W^u_{f,loc}(b)$. Differentiating this along the unstable direction at $b$ then yields the following relation:
$$
    D_uh(b)=D_u\hol^{s,g}_{h(a),h(b)}(h(a)D_uh(a)D_u\hol^{s,f}_{b,a}(b).
$$
Based on this, we choose the $\sigma_i$ so that 

\begin{equation}
    \label{eqn:Derivative Condition}
    D_uh_N(x_i)=D_u\hol^{s,g}_{h(x_0),h(x_i)}(h(x_0))D_uh_N(x_0)D_u\hol^{s,f}_{x_i,x_0}(x_i),
\end{equation}
where $x_i$ ranges over $x_0',$ $x_1'$, and $x_1$.
\begin{lemma}
\label{lem: Uniformity along A^u_f}
    The function $h_{N,\sigma}|_{A^u_f}$ is $C^2$ along $A^u_f$ and for every $0<\alpha<1$, the H\"older norm $||h_{N,\sigma}|_{A^u_f}||_{C^{1+\alpha}}$ is uniformly bounded for all $f,g\in\mathcal{U}$.
\end{lemma}
\noindent We defer the proof to Section \ref{sec: Uniformity of the New Conjugacy}.

For the sake of simplicity, we will drop the dependence on the $\sigma_i$ in our notation for the rest of this section and simply write $h_N$ in place of $h_{N,\sigma}$.
Our goal now will be to extend the definition of $h_N$ to a $C^{1+\alpha}_u$ function on all of $\mathbb{T}^2$. We will accomplish this using a weighted holonomy construction.

Let $x'\in\mathbb{T}^2/A^u_f$. Consider the points $y'$ and $z'$ which are obtained as the first intersections of $W^s_f(x')$ with $A^u_f$ in either direction (in Figure \ref{fig:Figure1}, $y'$ is obtained by moving $x'$ to the left, and $z'$ is obtained by moving $x'$ to the right). We define on $W^u_{f,loc}(x')$ the stable holonomy maps $y:W^u_{f,loc}(x')\rightarrow A^u_f$, $z:W^u_{f,loc}(x')\rightarrow A^u_f$ given by
$y(x)=\hol^{s,f}_{x',y'}(x)$ and $z(x)=\hol^{s,f}_{x',z'}(x)$. We can then apply $h_N$ to $y(x)$ and $z(x)$ and then take the respective stable holonomies from $A^u_g$ to $W^u_g(h(x))$. In general, the resulting two points will not be the same and we would like our point $h_N(x)$ to lie between these two points. To make this precise, we will define a $C^{1+Zyg}$ weight function 
$\rho:\mathbb{T}^2/A^u_g\rightarrow [0,1]$, and then define
\begin{equation}
\label{def:h_N}
 h_N(x)=\rho(h(x))\hol^{s,g}_{h(y(x)),h(x)}((h_N)|_{A^u_f}(y(x)))+(1-\rho(h(x)))\hol^{s,g}_{h(z(x)),h(x)}((h_N)|_{A^u_f}(z(x))),   
\end{equation}
where this convex combination is taken within the leaf relative to the natural Riemannian structure. We claim that for an appropriately chosen $\rho$, the resulting function $h_N$ is a continuous bijection, and hence a homeomorphism.

As a first observation, recall that $h_N(x_0)=h(x_0)$ and $h_N(x_1)=h(x_1)$. In particular, $h_N(x_0)$ and $h_N(x_1)$ lie on the same stable manifold. Now let $x\in A^s_f$. Then $y(x)=x_1,$ and so 
$$
\hol^{s,g}_{h(y(x)),h(x)}((h_N)|_{A^u_f}(y(x)))=\hol^{s,g}_{h(x_1),h(x)}((h_N)|_{A^u_f}(x_1)=\hol^{s,g}_{h(x_1),h(x)}(h(x_1))=h(x),
$$
and similarly $\hol^{s,g}_{h(z(x)),h(x)}((h_N)|_{A^u_f}(z(x)))=h(x)$. So regardless of the value of $\rho(h(x))$, we have
$$
h_N|_{A^s_f}=h|_{A^s_f}.
$$
Therefore we may actually allow discontinuities of $\rho\circ h$ as we cross $A^s_f$ along unstable leaves, so long as the one sided derivatives along unstable leaves match at $A^s_f$. We may also allow discontinuities of $\rho$ on the stable strips $[x_1,x_1']_s$ and $[x_0,x_0']_s$, but we will not need this for the construction.

In order for our mapping $h_N$ to be continuous, we require that as $x$ approaches $y(x)$ along its stable leaf, $\rho(h(x))\rightarrow 1$, and as 
$x$ approaches $z(x)$ along its stable leaf, $\rho(h(x))\rightarrow 0$. In fact, we will impose the stronger condition that $\rho(h(x))=1$ in a small neighborhood of $A^s_f$ on the ``right" side, and that $\rho(h(x))=0$ in a small neighborhood on the ``left" side. This is enough to ensure continuity except on the intervals $A^s_f$, $[x_1,x_1']_s$, and $[x_0,x_0']_s$. To see continuity at a point $x\in A^s_f$, let $x_n$ be a sequence of points approaching $x$ along $W^u(x)$ from ``below" (relative to Figure \ref{fig:Figure1}). In this case, $y(x_n)\rightarrow x_0'$ and $z(x_n)\rightarrow x_1$. By construction, $h_N(x_0')=h(x_0')$ and $h_N(x_1)=h(x_1)$ lie on the same stable manifold, so the stable holonomy to $W^u_g(h(x))$ takes both $h(x_0')$ and $h(x_1)$ to the same point, namely to $h(x)$:
$$
\lim_{n\rightarrow\infty}h_N(x_n)=\rho(h(x))h(x)+(1-\rho(h(x)))h(x)=h(x),
$$
exactly as required. If $x_n\rightarrow x$ from ``above," then 
$y(x_n)\rightarrow x_0$ and $z(x_n)\rightarrow x_1'$, and the remainder of the argument is the same. Likewise we can establish continuity for $x\in[x_1,x_1']_s$ and $x\in[x_0,x_0']_s$.

With this restriction on $\rho$, we have that $h_N$ as defined by \ref{def:h_N} is continuous on all of $\mathbb{T}^2$. We next need to impose conditions on $\rho$ to ensure that $h_N$ is a bijection. By construction, $h_N$ maps unstable leaves of $f$ to unstable leaves of $g$, so points on different stable manifolds are mapped to distinct points. Thus we only need to show that $h_N$ can be made to be a invertible when to restricted to each unstable manifold of $f$.

We would like to require $\rho(x)$ to be constant along unstable leaves of $g$, though this is not possible due to minimality of the unstable foliation. Instead, we will keep $\rho(x)$ constant on unstable manifolds of $g$, except for jump discontinuities when crossing $A^s_g$. Since the conjugacy $h$ preserves the unstable foliation, $\rho(h(x))$ will be constant on unstable manifolds of $f$, except for jump discontinuities when crossing $A^s_f$.

Referring to Figure \ref{fig:Figure1}, we consider the segment of stable manifold $B^s_f$ going from point $a$ to point $c$. As $x\rightarrow x_0$ from the left, we require that that $\rho(h(x))\rightarrow0$ and we similarly require that $\rho(h(x))\rightarrow1$ as $x\rightarrow x_1$ from the right. Therefore we should have $\rho(h(a))=0$ and $\rho(h(b))=1$. Define $\rho(x)$ on $B^s_f$ to be any $C^\infty$ function with $\rho(h(a))=0$, $\rho(h(b))=1$, and $\lim_{x\rightarrow c}\rho(h(x))=0$. In fact, we will impose the stronger requirement that $\rho\circ h$ is constant in sufficiently small neighborhoods of the points $a, b,$ and $c$ in the stable manifolds. We then extend $\rho(x)$ to all of $\mathbb{T}^2/(A^u_g\cup A^s_g)$ by defining it to be constant on unstable leaves up to the first intersection with $A^s_g$.
\begin{figure}
    \centering
    \includegraphics[width=0.5\linewidth]{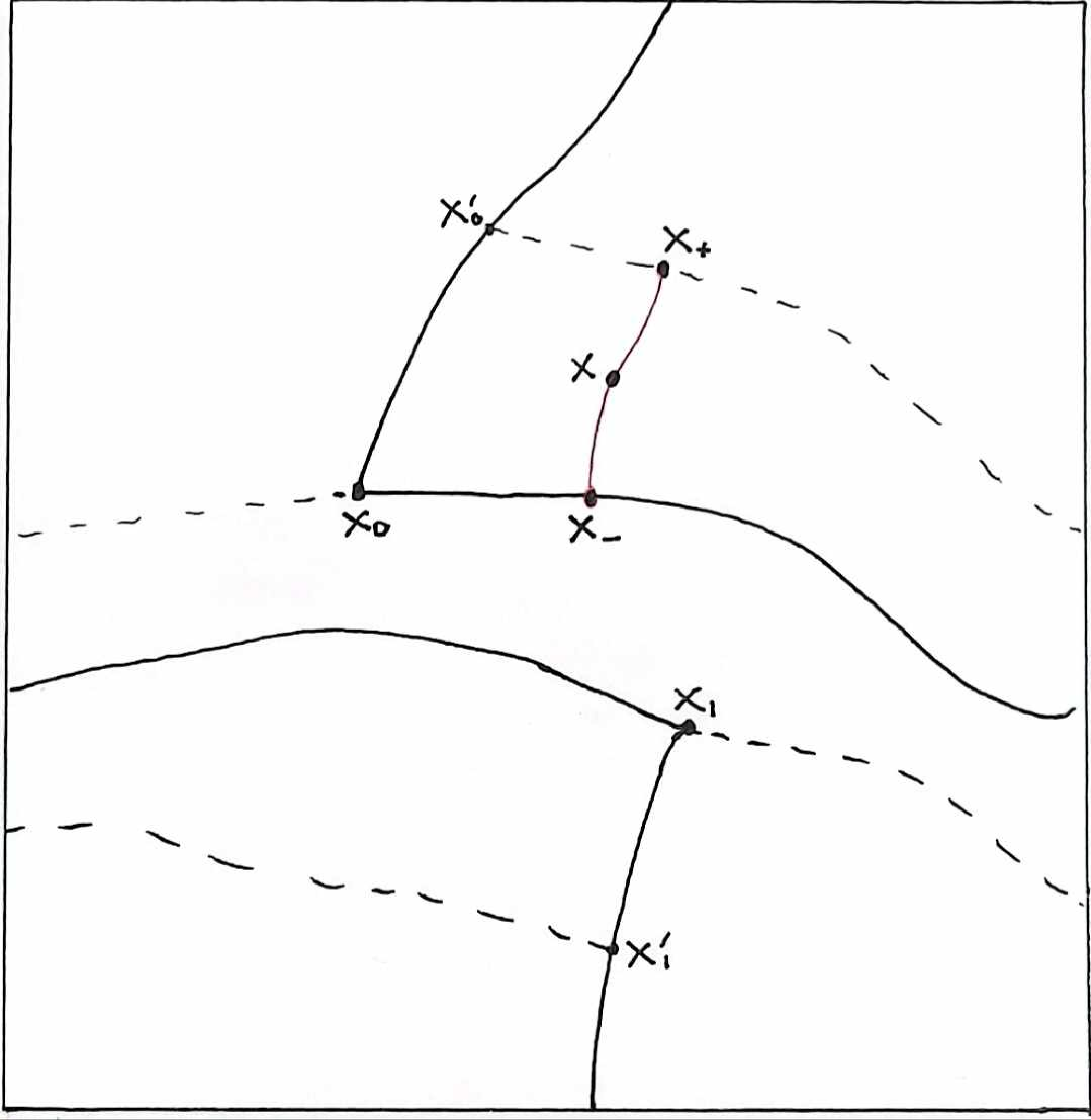}
    \caption{The segment $[x_-,x_+]_u$ shown in red.}
    \label{fig:Figure 2}
\end{figure}
\begin{lemma}
    The function $h_N$ as defined in \ref{def:h_N} with this choice of $\rho(x)$ is a bijection.
\end{lemma}
\begin{proof}
    Since $h_N$ maps unstable leaves of $f$ to unstable leaves of $g$, it is enough to show that $h_N$ is a bijection when restricted to any unstable leave of $f$. To this end, fix $x\in\mathbb{T}^2$. Since we know that $h_N|_{A^u_f}$ is a bijection, we may assume that $x\in\mathbb{T}^2/A^u_f.$ 
    Let $x_{-}$ and $x_{+}$ be the first two points (one in each direction) on $W^u_f(x)$ that lie on either $A^s_f$, $[x_0,x_0']_s$, or $[x_1,x_1']_s$, see Figure \ref{fig:Figure 2}. Then $h_N(x_-)=h(x_-)$ and $h_N(x_+)=h(x_+)$, and so by continuity, $h_N([x_-,x_+]_u)=[h(x_-),h(x_+)]_u$. By repeating the argument on the next unstable segments following $x_-$ and $x_+$, we find that $h_N$ is surjective. 

    To prove injectivity, it suffices to prove injectivity on the segment $[x_-,x_+]_u$. On the interior $(x_-,x_+)_u$, the holonomy functions $y(x)$ and $z(x)$ are continuous, and by construction, the average $\rho(h(x))$ is constant on $(x_-,x_+)_u$, say with value $\rho(h(x))=\rho$. Since both $h_N(y(x))$ and $h_N(z(x))$ move monotonically along $A^u_g$, the holonomies $\hol^{s,g}_{h(y(x_-)),h(x_-)}(h_N(y(x)))$ and 
    $\hol^{s,g}_{h(z(x_-)),h(x_-)}(h_N(z(x)))$ move monotonically along $(x_-,x_+)_u$. Therefore their weighted average by any constant weight (in particular by $\rho$) moves monotonically along $(x_-,x_+)_u$. Hence $h_N$ is injective.
    
\end{proof}

We will next establish differentiability of $h_N$ along unstable manifolds. Notice that although $h(x)$ appears in our definition of $h_N(x)$, as part of the averaging process $\rho(h(x))$, it will not limit the regularity in the unstable direction since $\rho(h(x))$ is constant along unstable leaves. Since the stable holonomy maps are $C^{1+Zyg}$ (see Section \ref{sec: Regularity of the Stable and Unstable Foliations}) and $h_N|_{A^u_f}$ is $C^2$, we have that $h_N|_{W^u_{f,loc}(x)}$ is $C^{1+Zyg}$ when $W^u_{f,loc}(x)$ does not intersect $A^s_f\cup[x_0,x_0']_s\cup[x_1,x_1']_s$. Thus it remains to establish differentiability of $h_N$ along unstable leaves at points on these stable segments. Consider $x\in A^s_f$ and let $x_n\rightarrow x$ from ``below" on $W^u_f(x)$. Recalling that $\rho(h(x_n))$ is constant, we have
$$
D_uh_N(x_n)=\rho(h(x_n))D_u\left(\hol^{s,g}_{h(y(x_n)),h(x_n)}\circ(h_N|_{A^u_f})\circ y\right)(x_n)+$$
$$
(1-\rho(h(x_n)))D_u\left(\hol^{s,g}_{h(z(x_n)),h(x_n)}\circ(h_N|_{A^u_f})\circ z\right)(x_n)=
$$
$$
\rho(h(x_n))\left(D_u\hol^{s,g}_{h(y(x_n)),h(x_n)}(h_N|_{A^u_f}(y(x_n)))
D_uh_N|_{A^u_f}(y(x_n))D_uy(x_n)\right)+
$$
$$
(1-\rho(h(x_n)))\left(D_u\hol^{s,g}_{h(z(x_n)),h(x_n)}(h_N|_{A^u_f}(z(x_n)))
D_uh_N|_{A^u_f}(z(x_n))D_uz(x_n)\right).
$$
Letting $\rho(h(x_{below})):=\rho(h(x_n))$ and letting $n\rightarrow\infty$ we have 
$$
D_uh_N(x)=\rho(h(x_{below}))\left(D_u\hol^{s,g}_{h(x_0'),h(x)}(h_N|_{A^u_f}(x_0'))
D_uh_N|_{A^u_f}(x_0')D_uy(x)\right)+
$$
$$
(1-\rho(h(x_{below})))\left(D_u\hol^{s,g}_{h(x_1),h(x)}(h_N|_{A^u_f}(x_1))
D_uh_N|_{A^u_f}(x_1)D_uz(x)\right):=
$$
\begin{equation}
    \label{eqn:Expression for Derivative from below}
   \rho(h(x_{below}))D_1(x)+(1-\rho(h(x_{below})))D_2(x). 
\end{equation}

When $x_n\rightarrow x$ from ``above" on $W^u(x)$, we similarly get
$$
D_uh_N(x)=\rho(h(x_{above}))\left(D_u\hol^{s,g}_{h(x_0),h(x)}(h_N|_{A^u_f}(x_0))
D_uh_N|_{A^u_f}(x_0)D_uy(x)\right)+
$$
$$
(1-\rho(h(x_{above})))\left(D_u\hol^{s,g}_{h(x_1'),h(x)}(h_N|_{A^u_f}(x_1'))
D_uh_N|_{A^u_f}(x_1')D_uz(x)\right):=
$$
\begin{equation}
    \label{eqn:Expression for Derivative from above}
    \rho(x_{above})D_3(x)+(1-\rho(x_{above}))D_4(x).
\end{equation}

In order to have continuity of the first derivative, we will need all of the $D_i(x)$ to be equal to a common value $D(x)$, $i=1,2,3,4$. 
This follows from (\ref{eqn:Derivative Condition}):
$$
D_1(x)=(D_u\hol^{s,g}_{h(x_0'),h(x)}(h(x_0')))
(D_uh_N|_{A^u_f}(x_0'))((D_u\hol^{s,f}_{x,x_0'})(x))=
$$
$$
(D_u\hol^{s,g}_{h(x_0'),h(x)}(h(x_0')))
\left(D_u\hol^{s,g}_{h(x_0),h(x_0')}(h(x_0))(D_uh_N|_{A^u_f}(x_0))(D_u\hol^{s,f}_{x_0',x_0}(x_0')) \right)(D_u\hol^{s,f}_{x,x_0'}(x))
$$
$$
=(D_u\hol^{s,g}_{h(x_0),h(x)}(h(x_0)))(D_uh_N|_{A^u_f}(x_0))
(D_u\hol^{s,f}_{x,x_0}(x))=D_3(x).
$$
Here we used the chain rule together with the composition properties $\hol^{s,f}_{x_0',x_0}=\hol^{s,f}_{x,x_0}\circ\hol^{s,f}_{x_0',x}$ and 
$\hol^{s,g}_{h(x_0),h(x_0')}=\hol^{s,g}_{h(x),h(x_0')}\circ\hol^{s,g}_{h(x_0),h(x)}$.
Completely identical calculations show that $D_2(x)=D_3(x)=D_4(x)$, as desired.

Finally, we need to check that $D_uh_N$ is well defined at the points $x_0$ and $x_1$. We begin by considering the $D_uh_N(b)$ as $b$ approaches $x_1$ from ``above" along $W^u_f(x_1)$ (see Figure \ref{fig:Figure1}). Since $\rho(h(b))=1$ we have
$$
D_uh_N(b)=D_u\hol^{s,g}_{h(x_0),h(x_1)}(h(x_1))D_uh_N(\hol^{s,f}_{x_1,x_0}(b))D_u\hol^{s,f}_{x_1,x_0}(b)
$$
Letting $b\rightarrow x_1$ we get the requirement
$$
D_uh_N(x_1)=D_u\hol^{s,g}_{h(x_0),h(x_1)}(h(x_1))D_uh_N(x_0)D_u\hol^{s,f}_{x_1,x_0}(x_1),
$$
which is exactly equation (\ref{eqn:Derivative Condition}). Using the identity
$$
D_uh_N(x_0)=(D_u\hol^{s,g}_{h(x_0),h(x_1)}(h(x_1)))^{-1}D_uh_N(x_0)(D_u\hol^{s,f}_{x_1,x_0}(x_1))^{-1}=
$$
$$
D_u\hol^{s,g}_{h(x_1),h(x_0)}(h(x_1))D_uh_N(x_0)D_u\hol^{s,f}_{x_0,x_1}(x_0),
$$
we are similarly able to establish differentiability at $x_0$. This establishes differentiability of $h_N$ on all of $\mathbb{T}^2$.
\begin{lemma}
\label{lem:Uniform bound for unstable derivative first step}
    For every $f,g\in\mathcal{U}$, $h_N\in C^{1+\alpha}_u$. Moreover, $||D_uh_N||_{\alpha}$ is uniformly bounded in $f,g\in\mathcal{U}$.
\end{lemma}
\noindent We will postpone the proof of the uniformity statement to Section \ref{sec: Uniformity of the New Conjugacy}.

Notice that although $h_N$ preserves the unstable foliations, it does not preserve the stable foliation; that is, $h_N$ does not in general send stable leaves of $f$ to stable leaves of $g$. However, we claim that $h_N$ maps the stable foliation to a $C^{1+\alpha}$ foliation with $C^{2}$ leaves. To make this precise, let 
$\Tilde{W}^s_g(x):=h_N(W^s_f(h^{-1}_N(x)))$, and let $\Tilde{\mathcal{F}}^{s,g}$ denote the foliation by the leaves $\Tilde{W}^s_g(x)$.
\begin{lemma}
\label{lem: Fake Stable Foliation}
    For every $g\in\mathcal{U}$, $\Tilde{\mathcal{F}}^{s,g}$ is a $C^{1+\alpha}$ foliation consisting of $C^2$ leaves.
\end{lemma}
\begin{proof}
    Fix $x\in\mathbb{T}^2/A^u_f$, and consider the local stable manifold through $x$, $W^s_{f,\delta}(x)$, where $\delta>0$ is such that $W^s_{f,\delta}(x)\cap A^u_f=\emptyset$. For every $x'\in W^s_{f,\delta}(x)$, $y:=y(x')$ and $z:=z(x')$ are constant. Therefore, $\hol^{s,g}_{h(y),h(x')}(h_N(y))$ and $\hol^{s,g}_{h(z),h(x')}(h_N(z))$ trace out segments of the stable manifolds $W^s_g(h_N(y))$ and $W^s_g(h_N(z))$, respectively. Since $g$ is $C^2$, these submanifolds are also $C^2$. However, they are parameterized by the H\"older continuous function $h(x)$, which accounts for $h_N(x)$ being only H\"older continuous in the stable direction. Then $h_N(x)$ is attained by averaging $W^s_g(h_N(y))$ and $W^s_g(h_N(z))$ along $W^u_g(h(x))$ with respect to the function $\rho(h(x))$. The function $\rho(x)$ is $C^2$ along stable leaves of $g$, but is also parameterized by the H\"older continuous function $h(x)$. Therefore, $h_N(W^s_{f,\delta}(x))$ can be regarded, after a reparameterization by $h^{-1}$, as a $C^2$ of $\rho$ over the $C^2$ manifold $W^s_g(h_N(y))$. Moreover, since $\rho(x)= 1$ sufficiently close to $h_N(y)$ on the left (relative to \ref{fig:Figure1}) and $\rho(x)= 0$ sufficiently close to $h_N(z)$ on the right, 
    $\Tilde{W}^s_{g,\delta}(x)=W^s_{g,\delta}(x)$ 
    for $x\in A^u_g$ and $\delta>0$ sufficiently small. Therefore, $\Tilde{\mathcal{F}}^{s,g}$ consists of $C^2$ leaves.

    Finally, since $h_N\in C^{1+\alpha}_u$ by Lemma \ref{lem:Uniform bound for unstable derivative first step}, it maps the stable foliation of $f$ to a foliation that is $C^{1+\alpha}$ transversal to the unstable foliation of $g$.
\end{proof}
The homeomorphism $h_N$ is not differentiable in the stable direction. This is easy to see, since $h_N|_{A^s_f}=h|_{A^s_f}$, and our original conjugacy is assumed to not be $C^1_s$ (or else there would be nothing to prove). The next step in the construction is to repeat the previous phase in the stable direction, using $h_N$ in place of $h$, as we now explain.

From the conjugacy equation $hf=gh$ we get $hf^{-1}=g^{-1}h$. Observe that $f^{-1}$ (and likewise $g^{-1}$) is an Anosov diffeomorphism whose stable and unstable manifolds are switched from those of $f$: $W^u_f(x)=W^s_{f^{-1}}(x)$. In particular, $A^s_f$ is a segment of unstable manifold for $f^{-1}$. Let $\mu_{f^{-1}}$ and $\mu_{g^{-1}}$ be the SRB measures of $f^{-1}$ and $g^{-1}$ respectively. Then if $h$ matched all of the stable periodic data for $f$ and $g$, we would have $h_*\mu_{f^{-1}}=\mu_{g^{-1}}$, from which it follows that $h$ pushes forward the conditional measures along unstable leaves of $f^{-1}$ to the corresponding conditional measures of $\mu_{g^{-1}}$. Then, as before, we get the representation
$$
h=(I^s_{g,h(x_0)})^{-1}\circ I^s_{f,x_0}.
$$
As in the first step of the construction, we break the segment $A^s_f$ into three subsegments: $A^s_{f,1}=[x_0,x_0'']_s$,
$A^s_{f,2}=[x_0'',x_1'']_s$, and $A^s_{f,1}=[x_1'',x_1]_s$. We then define
$$
\overline{h}_N|_{A^s_{f,1}}(x):=(I^s_{g,h(x_0)})^{-1}\circ I^s_{f,x_0}(x),
$$
for appropriately normalized conditional measures, and we define 
$\overline{h}_N|_{A^s_{f,2}}$ and $\overline{h}_N|_{A^s_{f,3}}$ in the same exact way. After applying $C^2$ reparameterizations 
$\overline{\sigma}_i:A^s_{f,i}\rightarrow A^s_{f,1}$,
we obtain a $C^2$ diffeomorphism $\overline{h}_N|_{A^s_f}:A^s_f\rightarrow A^s_f$ which agrees with $h$ at the points $x_0, x_0'', x_1''$, and $x_1$. 

As before we want to extend the domain of $\overline{h}_N|_{A^s_f}$ to all of $\mathbb{T}^2$ via a weighted holonomy construction. Given a point $x\in\mathbb{T}^2$, we have two choices of unstable holonomy to make leading to points $\overline{y}(x),\overline{z}(x)\in A^s_f$ (in Figure \ref{fig:Figure1} we think of $\overline{y}(x)$ as going ``up" and $\overline{z}(x)$ as going ``down"). Then we apply $\overline{h}_N|_{A^s_f}$ to each of these points and take the unstable holonomy back to $\Tilde{W}^s_g(h_N(x))$, instead of to $W^s_g(h(x))$. Taking a weight function
$\overline{\rho}:\mathbb{T}^2/A^s_f\rightarrow [0,1]$, we then define
$$
\overline{h}_N(x)=\overline{\rho}(h_N(x))\Tilde{\hol}^{u,g}_{h(\overline{y}(x)),h_N(x)}((\overline{h}_N)|_{A^s_f}(\overline{y}(x)))+
$$
\begin{equation}
\label{eqn: Definition of second step Construction}
    (1-\overline{\rho}(h_N(x)))\Tilde{\hol}^{u,g}_{h(\overline{z}(x)),h_N(x)}((\overline{h}_N)|_{A^s_f}(\overline{z}(x))),
\end{equation}
where $\Tilde{\hol}^{u,g}$ denotes the unstable holonomy between leaves of the foliation $\Tilde{\mathcal{F}}^{s,g}$.

This definition forces $\overline{h}_N|_{A^u_f}=h_N|_{A^u_f}$. As before, we may allow discontinuities of $\overline{\rho}$ as it crosses $A^u_g$ and so we define it to be constant along stable leaves of $f$ up until their first intersection with $A^u_f$. For continuity, we require that $\overline{\rho}(x)\rightarrow 1$ as $x$ approaches $A^s_f$ from ``below" and $\overline{\rho}(x)\rightarrow 0$ as $x$ approaches $A^s_f$ from ``above." For continuity of the derivative, it will also be important for us to require that $D_u\overline{\rho}(x)\rightarrow 0$ as $x$ approaches $A^s_f$ from either direction. In fact, we will assume the stronger condition that $\overline{\rho}(x)$ is identically $1$ in a small neighborhood ``below" $A^s_f$ and identically equal to $0$ in a small neighborhood ``above" $A^s_f$. Analogously to Lemma \ref{lem:Uniform bound for unstable derivative first step}, the conjugacy $\overline{h}_N$ is uniformly smooth in the stable direction:

\begin{lemma}
\label{lem:s-diff of second step}
    For every $f,g\in\mathcal{U}$, $\overline{h}_N\in C^{1+\alpha}_s$. Moreover, there exists $M_1>0$ such that every $f,g\in\mathcal{U}$ and every $x\in\mathbb{T}^2$, $||\overline{h}_N|_{W^s_f(x)}||_{C^{1+\alpha}}\leq M_1$. 
\end{lemma}

The next step is to show that we maintained differentiability in the unstable direction with this construction and have $\overline{h}_N\in C^{1+\alpha}_u$.

\begin{lemma}
\label{lem: u-diff Second Step}
    For every $f,g\in\mathcal{U}$, $\overline{h}_N\in C^{1+\alpha}_u$. Moreover, there exists $M_2>0$ such that every $f,g\in\mathcal{U}$ and every $x\in\mathbb{T}^2$, $||\overline{h}_N|_{W^u_f(x)}||_{C^{1+\alpha}}\leq M_2$.
\end{lemma}
We defer the proofs of Lemmas \ref{lem:s-diff of second step} and \ref{lem: u-diff Second Step} to the next subsection.
Combining these lemmas together with the Journ\'e lemma, we have the following:
\begin{theorem}
\label{Thm:Uniform Conjugacy Bounds}
    For every $f,g\in\mathcal{U}$, $\overline{h}_N\in C^{1+\alpha}$. Moreover, there exists $M>0$ independent of $N\in\mathbb{N}$ and $f,g\in\mathcal{U}$ such that $||\overline{h}_N||_{C^{1+\alpha}}\leq M$.
\end{theorem}

\subsection{Uniformity of the New Conjugacy}
\label{sec: Uniformity of the New Conjugacy}
In this subsection we proof the uniformity claims of the previous section.
\begin{proof}[Proof of Lemma \ref{lem: Uniformity along A^u_f}]
    By the chain rule, for $x\in [x_0,x_0']_u$,
    \begin{equation}
        \label{eqn: Formula for h_N'}
        h_N'(x)=((I^u_{g,h(x_0)})^{-1})'(I^u_{f,x_0}(x))((I^u_{f,x_0})'(x))=
        \frac{\omega^u_f(x)}{\omega^u_g((I^u_{g,h(x_0)})^{-1}(I^u_{f,x_0}(x)))}=
        \frac{\omega^u_f(x)}{\omega^u_g(h_N(x))},
    \end{equation}
    where $\omega^u_f$ and $\omega^u_g$ are the normalized densities of $\mu^u_{f,x_0}$ and $\mu^u_{g,h(x_0)}$, respectively. Let
    $$
    \omega^u_f(x,x_0)=\prod_{i=1}^\infty \frac{D_uf(f^{-n}(x))}{D_uf(f^{-n}(x_0))}
    $$
    be the density of the conditional measure normalized by the condition $\omega^u_f(x_0,x_0)=1$. Then
    $$
    \omega^u_f(x)=\frac{\omega^u_f(x,x_0)}{\int_{x_0}^{x_0'}\omega^u_f(z,x_0)dm(z)},
    $$
    and similarly for $\omega^u_g(x)$. We first uniformly bound the density $\omega^u_f(x_0,x_0)$ in terms of the hyperbolicity rates (\ref{eqn:Anosov Condition}) and the length of the interval $[x_0,x_0']_u$. We begin by taking logarithms
    $$
    \log \omega^u_f(x,x_0)=\sum_{i=1}^\infty(\log D_uf(f^{-n}(x))-\log D_uf(f^{-n}(x_0)))\leq
    $$
    $$
    \sum_{i=1}^\infty|\log D_uf|_{C^1}d^u(f^{-n}(x),f^{-n}(x_0))\leq
    \frac{|f|_{C^2}}{\mu_1}\sum_{i=1}^\infty\nu_1^{-n}d^u(x,x_0)=Kd^u(x,x_0),
    $$
    with $K>0$ uniform in $f\in\mathcal{U}$. Thus $\omega^u_f(x,x_0)\leq e^{Kd^u(x,x_0)}$, by a symmetric argument with $\omega^u_f(x,x_0)^{-1}$ gives the lower bound
    $\omega^u_f(x,x_0)\geq e^{-Kd^u(x,x_0)}$. Therefore
    \begin{equation}
    \label{eqn:Uniform Bound on Density}    
    \omega^u_f(x)\leq \frac{e^{Kd^u(x_0,x_0')}}{\int_{x_0}^{x_0'}e^{-Kd^u(x_0,x_0')}dm(z)}=\frac{e^{2Kd^u(x_0,x_0')}}{d^u(x_0,x_0')}.
    \end{equation}
    We next must establish uniform upper and lower bounds on the length of $[x_0,x_0']_u$ for different choices of $f$. This will be done by establishing $C^1$ local uniformity and using the Arzel\`a-Ascoli argument. Namely, given $\varepsilon>0$ sufficiently small, by structural stability there exists $\delta>0$ such that if $d_{C^1}(f,\overline{f})<\delta$, then $f$ and $\overline{f}$ are conjugate by a homeomorphism $h_{f,\overline{f}}$ satisfying
    $d_{C^0}(h_{f,\overline{f}},\id)+d_{C^0}(h_{f,\overline{f}}^{-1},\id)<\varepsilon$. Moreover, by Lemma \ref{lem:Uniform Holder}, the H\"older exponent and seminorms for $h_{f,\overline{f}}$ and $h_{f,\overline{f}}^{-1}$ are uniformly bounded. We let $A^u_{\overline{f},i}=h_{f,\overline{f}}(A^u_{f,i})$ for $i=1,2,3$ and we have
    \begin{equation}
        \label{eqn: Length Upper bound}
    d(h_{f,\overline{f}}(x_0),h_{f,\overline{f}}(x_0'))\leq
    |h_{f,\overline{f}}|_{\alpha_0}d(x_0,x_0')^{\alpha_0}.    
    \end{equation}
    For a lower bound, we have
    $$
    d(x_0,x_0')\leq |h_{f,\overline{f}}^{-1}|_{\alpha_0}d(h_{f,\overline{f}}(x_0),h_{f,\overline{f}}(x_0'))^{\alpha_0},
    $$
    which gives
    \begin{equation}
        \label{eqn: Length Lower Bound}
    |h_{f,\overline{f}}^{-1}|^{-1/\alpha_0}_{\alpha_0}d(x_0,x_0')^{1/\alpha_0}\leq d(h_{f,\overline{f}}(x_0),h_{f,\overline{f}}(x_0')).
    \end{equation}
    This proves that the lengths of the $A^u_{\overline{f},i}$ can be made uniform in a $C^1$ neighborhood of $f$.
    Then by (\ref{eqn:Uniform Bound on Density}), we have a uniform upper bound on $\omega^u_f(x)$ in a sufficiently small $C^1$ neighborhood of $f$. We likely get a uniform lower bound for $\omega^u_f(x)$ as well as the same bounds for $\omega^u_g$. Therefore, by (\ref{eqn: Formula for h_N'}), $h_N'(x)$ is uniformly bounded both above and below for all $x\in A^u_{f,i}$ and so the difference between $h_N'(x_0')$ from each side is also uniformly bounded. Recall that the derivatives at the endpoints of the intervals $A^u_{f,i}$ must be given by (\ref{eqn:Derivative Condition}). The derivative of the stable holonomy is given by the formula
    \begin{equation}
        \label{eqn:Derivative of Holonomy Formula}
        D_u\hol^{s,f}_{a,b}(x)=\prod_{n=0}^\infty\frac{D_uf(f^n(x))}{D_uf(f^n(\hol^{s,f}_{a,b}(x)))}.
    \end{equation}
    The same arguments used to establish uniform bounds for $\omega^u_f(x,x_0)$ also establish uniform bounds on $D_u\hol^{s,f}_{x_i,x_0}(x_i)$, where $x_i$ ranges over $x_0',x_1',$ and $x_1$, and likewise for $D_u\hol^{s,g}$. Therefore the $C^1$ sizes of the reparameterizations $\sigma_i$ needed to satisfy condition (\ref{eqn:Derivative Condition}) can be made uniform in a $C^1$ of $f$ and $g$ in $\mathcal{U}\times\mathcal{U}$. By the Arzel\`a-Ascoli argument, this gives a uniform bound for all $f,g\in\mathcal{U}$.This establishes uniform bounds on the $C^1$ norm of $h_{N,\sigma}|_{A^u_f}$. However, to establish $C^2$ bounds on $h_{N,\sigma}|_{A^u_f}$, we will need to show that the $C^1$ norm of the reparameterizations $\sigma_i$ can be made small. We will prove the following:
    \begin{sublemma}
        \label{lem: C1 size of reparam}
        There exist constants $C>0$, $0<\gamma<1$ such that $|\sigma_i|_{C^1}\leq C\gamma^N$.
    \end{sublemma}
    \begin{proof}
        For concreteness, we focus on $\sigma_1$. First observe that 
        \begin{equation*}
            D_uh_N(x_0)=\frac{\int_{h(x_0)}^{h(x_0')}d\mu^u_{g,h(x_0)}}{\int_{x_0}^{x_0'}d\mu^u_{f,x_0}}.
        \end{equation*}
           Since the integral $\int_{x_0}^{x_0'}d\mu^u_{f,x_0}$ is uniformly bounded from below, this derivative will converge to $1$ exponentially as $N$ tends to infinity by the same effective equidistribution argument we will present in Section \ref{sec: Proof the Main Theorem the C^0 Estimates} (see Theorem \ref{Thm:Effective Equidistribution}). By condition (\ref{eqn:Derivative Condition}), we must have
           \begin{equation*}
               D_uh_N(x_0)=D_u\hol^{s,g}_{h(x_0'),h(x_0)}(h(x_0))D_uh_N(x_0')D_u\hol^{s,f}_{x_0,x_0'}(x_0)=
           \end{equation*}
           \begin{equation*}
               \frac{\int_{h(x_0)}^{h(x_0')}d\mu^u_{g,h(x_0)}}{\int_{x_0}^{x_0'}d\mu^u_{f,x_0}}\prod_{i=-\infty}^\infty \frac{D_uf(f^i(x_0)}{D_ug(h(f^i(x_0)))}\frac{D_ug(h(f^i(x_0')))}{D_uf(f^i(x_0')}
           \end{equation*}
            Therefore in order to show that $|\sigma_i|_{C^1}\leq C\gamma^N$ can be made $C^1$ small it suffices to show that
            \begin{equation}
            \label{eqn: Infinite product Converges Exponential Speed}
                \left|\prod_{i=-\infty}^\infty \frac{D_uf(f^i(x_0)}{D_ug(h(f^i(x_0)))}\frac{D_ug(h(f^i(x_0')))}{D_uf(f^i(x_0')} -1\right|\leq C\gamma^N
            \end{equation}
           Let $\varphi=\log D_uf-\log D_ug\circ h$. Then since $x_0$ is a fixed point of $f$, we have by the assumption on the periodic data that $\varphi(x_0)=0$. Taking logarithms, it then suffices to prove that
        \begin{equation*}
        \sum_{i=-\infty}^\infty\varphi(f^i(x_0'))\leq C\gamma^N.
        \end{equation*}
            Suppose for convenience that $N$ is odd, and consider the orbit segment of length $N$
        \begin{equation*}
            \{f^i(x_0')\}_{-\frac{N-1}{2}}^{\frac{N-1}{2}}.
        \end{equation*}
        By hyperbolicity, we have 
        \begin{equation*}
            d(x_0,f^{\frac{N-1}{2}}(x_0'))=d(f^{\frac{N-1}{2}}(x_0),f^{\frac{N-1}{2}}(x_0'))\leq d^u_f(f^{\frac{N-1}{2}}(x_0),f^{\frac{N-1}{2}}(x_0'))\leq C_{\mathcal{U}}\mu_1^{\frac{N-1}{2}}d^u_f(x_0,x_0')\leq C_1\gamma_0^N,
        \end{equation*}
        and likewise
        \begin{equation}
            d(x_0,f^{-\frac{N-1}{2}}(x_0'))\leq C_1\gamma_0^N
        \end{equation}
        for some uniform $C_1>0$ and $0<\gamma_0<1$. Therefore, $ \{f^i(x_0')\}_{-\frac{N-1}{2}}^{\frac{N-1}{2}}$ is a $2C_1\gamma_0^N$ pseudo-orbit, so for sufficiently large $N$ to apply shadowing, we have a point $x_N\in\fix(f^N)$ and a (uniform) constant $C_2>0$ such that for every $i=-\frac{N-1}{2},\cdots,\frac{N-1}{2}$, we have
        \begin{equation*}
            d(f^i(x_0'),f^i(x_N))\leq C_2C_1\gamma_0^N.
        \end{equation*}
        We now write
        \begin{equation*}
            \sum_{i=-\infty}^\infty\varphi(f^i(x_0'))=
            \sum_{i=-\frac{N-1}{2}}^{\frac{N-1}{2}}\varphi(f^i(x_0'))+
            \sum_{|i|>\frac{N-1}{2}}\varphi(f^i(x_0'))=
        \end{equation*}
        \begin{equation*}
             \sum_{i=-\frac{N-1}{2}}^{\frac{N-1}{2}}[\varphi(f^i(x_0'))-\varphi(f^i(x_N))]+ \sum_{i=-\frac{N-1}{2}}^{\frac{N-1}{2}}\varphi(f^i(x_N))+ \sum_{|i|>\frac{N-1}{2}}[\varphi(f^i(x_0'))-\varphi(f^i(x_0))].
        \end{equation*}
We estimate these three sums separately.

For the first sum, since $\varphi$ is $\alpha$-H\"older with $|\varphi|_\alpha$ uniformly bounded, we get by shadowing
\begin{equation*}
    \sum_{i=-\frac{N-1}{2}}^{\frac{N-1}{2}}[\varphi(f^i(x_0'))-\varphi(f^i(x_N))]\leq |\varphi|_\alpha \sum_{i=-\frac{N-1}{2}}^{\frac{N-1}{2}}d(f^i(x_0'),f^i(x_N))^\alpha\leq
\end{equation*}
\begin{equation*}
    |\varphi|_\alpha \sum_{i=-\frac{N-1}{2}}^{\frac{N-1}{2}}(C_2C_1)^\alpha\gamma_0^{\alpha N}= |\varphi|_\alpha (C_2C_1)^\alpha N\gamma_0^{\alpha N}\leq C_3\gamma^N,
\end{equation*}
where $\gamma_0^\alpha<\gamma<1$.
The second sum is identically equal to $0$ by the periodic data assumption. Finally, for the third sum, notice that the points $f^i(x_0)$ and $f^i(x_0')$ become exponentially close as $|i|\rightarrow\infty$. Considering (for simplicity of notation) forward iterates, we have
\begin{equation*}
    \sum_{i>\frac{N-1}{2}}[\varphi(f^i(x_0'))-\varphi(f^i(x_0))]\leq
    |\varphi|_\alpha\sum_{i>\frac{N-1}{2}}d(f^i(x_0'),f^i(x_0))=
\end{equation*}
\begin{equation*}
     |\varphi|_\alpha\sum_{i=1}^\infty d(f^{i+\frac{N-1}{2}}(x_0'),f^{i+\frac{N-1}{2}}(x_0))^\alpha\leq 
     C|\varphi|_\alpha\gamma_0^{\alpha N}\sum_{i=1}^\infty d(f^i(x_0'),f^i(x_0))^\alpha\leq C_4\gamma_0^{\alpha N}<C_4\gamma^N.
\end{equation*}
Combining all constants proves the sublemma.
    \end{proof}
    
    We now choose the the $C^0$-size of our reparameterization to be $\gamma^{2N}$-small; that is, $|\sigma|_{C^0}\leq \gamma^{2N}$. Together with Sublemma \ref{lem: C1 size of reparam}, this gives us $|\sigma_i|_{C^2}\leq C$ for some uniform $C>0$, or $|\sigma_i|_{C^{1+\alpha}}\leq C_\alpha$ for any $0<\alpha<1$ with $C_\alpha>0$ depending only on $\alpha$.
    
    It remains to uniformly bound the H\"older seminorm of the derivative of $h_N|{A^u_{f,i}}$ (without reparameterization). This will follow from showing that the second derivative on each $A^u_{f,i}$ can be made uniformly bounded.
    To establish $C^2$ bounds on $h_N$, it suffices to obtain $C^1$ upper bounds on $\omega^u_f$ and $\omega^u_g$ and use the quotient rule along with (\ref{eqn: Formula for h_N'}). It is enough for us to bound $\omega^u_f(x,x_0)$. Formally differentiating $\log\omega^u_f(x,x_0)$, we get
    \begin{equation}
        \label{eqn: log derivative of density}
        \frac{D_u\omega^u_f(x,x_0)}{\omega^u_f(x,x_0)}=
        \sum_{i=1}^\infty\frac{D^2_uf(f^{-n}(x))}{D_uf(f^{-n}(x))}D_uf^{-n}(x)
    \end{equation}
    Since $|D_uf^{-n}(x)|\leq \mu_1^{-n}$, 
    $$
    \left|\sum_{i=1}^\infty\frac{D^2_uf(f^{-n}(x))}{D_uf(f^{-n}(x))}D_uf^{-n}(x)\right|\leq
    \sum_{i=1}^\infty\frac{|D^2_uf(f^{-n}(x))|}{|D_uf(f^{-n}(x))|}|D_uf^{-n}(x)|\leq
    $$
    $$
    \sum_{i=1}^\infty\frac{|f|_{C^2}}{\mu_1}\mu_1^{-n}<\frac{|f|_{C^2}}{1-\mu_1}<\infty.
    $$
    By the Weierstrass M-test, the series converges uniformly to the derivative of $\log\omega^u_f(x,x_0)$, and moreover the bound on the logarithmic derivative of $\omega^u_f(x,x_0)$ is uniform in $f\in\mathcal{U}$. Together with our bound (\ref{eqn:Uniform Bound on Density}), this gives the desired bound on $D_u\omega^u_f$. This gives uniform $C^2$ bounds on $h_{N,\sigma}|_{A^u_f}$, which consequentially gives uniform $C^{1+\alpha}$ bounds for all $0<\alpha<1$.   
\end{proof}

\begin{proof}[Proof of Lemma \ref{lem:Uniform bound for unstable derivative first step}]
    Fix a point $x\in\mathbb{T}^2$. Since $\rho$ is constant along unstable manifolds we have
    $$
    D_uh_N(x)=\rho(h(x))D_u\hol^{s,g}_{h(y(x)),h(x)}(h_N(y(x)))D_uh_N|_{A^u_f}(y(x))D_uy(x)+$$
    $$
    (1-\rho(h(x)))D_u\hol^{s,g}_{h(z(x)),h(x)}(h_N(z(x)))D_uh_N|_{A^u_f}(z(x))D_uz(x).
    $$
    By Lemma \ref{lem: Uniformity along A^u_f}, $|D_uh_N|_{A^u_f}|_{C^{1+\alpha}}$ is uniformly bounded for all $0<\alpha<1$. Moreover, we established from (\ref{eqn:Derivative of Holonomy Formula}) that $D_u\hol^{s,f}_{a,b}$ is uniformly bounded, so long as $d^s_f(x,\hol^{s,f}_{a,b}(x))$ stays bounded. In other words, the maximum distance between a point $x\in\mathbb{T}^2$ and either $y(x)$ and $z(x)$ must be bounded. This is clearly bounded for a fixed $f$, and is true $C^1$ close to $f$ by structural stability and Lemma \ref{lem:Uniform Holder}. Thus by the Arzel\`a-Ascoli argument, it is true for all $f\in\mathcal{U}$. 

    Next we show that, restricted to $W^u_{f,loc}(x)$, the $\alpha$-H\"older seminorm of $D_u\hol^{s,f}_{a,b}$ is uniformly bounded for any $0<\alpha<1$. Note that it is sufficient to prove that $D_u\hol^{s,f}_{a,b}$ is locally H\"older. To be precise, suppose that $x\not\in [x_0,x_1]_s$ and let $\delta>0$ be small enough so that $W^u_{f,\delta}(x)$ does not intersect $[x_0,x_1]_s$. Let $x'\in W^u_{f,\delta}(x)$. Then we write
    $$
    \log(D_u\hol^{s,f}_{a,b}(x))-\log(D_u\hol^{s,f}_{a,b}(x'))
    =\sum_{i=0}^{M-1}\left(\log(D_uf(f^i(x)))-\log(D_uf(f^i(x'))) \right)
    $$
    $$
    +\sum_{i=0}^{M-1}\left(\log(D_uf(f^i(\hol^{s,f}_{a,b}(x))))-\log(D_uf(f^i(\hol^{s,f}_{a,b}(x)))) \right)
    $$
    $$+
    \sum_{i=M}^{\infty}\left((\log(D_uf(f^i(x)))-\log(D_uf(f^i(\hol^{s,f}_{a,b}(x)))) \right)
    $$
    \begin{equation}
          \label{eqn: Derivative of Holonomy is Holder Calculation}
    +\sum_{i=M}^{\infty}\left((\log(D_uf(f^i(x')))-\log(D_uf(f^i(\hol^{s,f}_{a,b}(x')))) \right),
    \end{equation}
    where $M:=M(d^u_f(x,x'))$ will chosen to optimize the two different type of terms in (\ref{eqn: Derivative of Holonomy is Holder Calculation}). Since $x$ and $x'$ lie on the same unstable manifold, the distance between them is expanded under forward iterates of $f$. We estimate the first and second sums in (\ref{eqn: Derivative of Holonomy is Holder Calculation}) as follows:
    $$
    \sum_{i=0}^{M-1}\left|\log(D_uf(f^i(x)))-\log(D_uf(f^i(x'))) \right|\leq \sum_{i=0}^{M-1}\frac{1}{\mu_1}d^u_f(f^i(x),f^i(x'))\leq 
    $$
    $$
    \sum_{i=0}^{M-1}\frac{1}{\mu_1}\nu_1^id^u_f(x,x')= \frac{\nu^M_1-1}{\mu_1(\nu_1-1)}d^u_f(x,x')<C\nu_1^Md^u_f(x,x').
    $$
    We estimate the third and fourth sums similarly:
    $$
    \sum_{i=M}^{\infty}\left|(\log(D_uf(f^i(x)))-\log(D_uf(f^i(\hol^{s,f}_{a,b}(x)))) \right|\leq
    \sum_{i=M}^\infty \frac{1}{\mu_1}\nu_0^id^s_f(x,\hol^{s,f}_{a,b}(x))=$$
    $$
    \frac{1}{\mu_1}\nu_0^Md^s_f(x,\hol^{s,f}_{a,b}(x))\sum_{i-0}^\infty\nu_0^i<C\nu_0^M.
    $$
    We may therefore estimate (\ref{eqn: Derivative of Holonomy is Holder Calculation}) as
    \begin{equation}
        \label{eqn:Choosing M to optimize Estimate}
        C(\nu_0^M+\nu_1^Md^u_f(x,x')),
    \end{equation}
    where $C>0$ is uniform. Let $0<\beta\leq 1$ be such that $\nu_0^\beta\leq \nu_1^{-1}$. Then for a given $0<\alpha<1$, we choose $M\geq 0$ to be minimum such that
    $$
    \nu_1^M\leq d^u_f(x,x')^{-\alpha}<\nu_1^{M+1},
    $$
    if such $M\in\mathbb{N}$ exists, and $M=0$ otherwise. Then we estimate
    $$
    C(\nu_0^M+\nu_1^Md^u_f(x,x'))\leq C(d^u_f(x,x')^{\alpha\beta}+d^u_f(x,x')^{1-\alpha})\leq 2Cd^u_f(x,x')^{\min\{\alpha\beta,1-\alpha\}},
    $$
    as desired. It remains to show that $D_uh_N$ is H\"older continuous when restricted to stable manifolds. Let $D_1(x)$ and $D_2(x)$ be as in (\ref{eqn:Expression for Derivative from below}). Then for $x'\in W^s_{f,\delta}(x)$ we have
    $$
    |D_uh_N(x)-D_uh_N(x')|=
    |\rho(h(x))D_1(x)+(1-\rho(x))D_2(x)-\rho(h(x'))D_1(x')-(1-\rho(x'))D_2(x')|
    $$
    $$
    \leq \rho(h(x))|D_1(x)-D_1(x')|+|D_1(x')||\rho(h(x))-\rho(h(x'))|+
    $$
    \begin{equation}
    \label{eqn:Transverse regularity First Step of Construction}
         (1-\rho(h(x)))|D_2(x)-D_2(x')|+|D_2(x')||\rho(h(x))-\rho(h(x'))|.
    \end{equation}    
    Since the lengths of the stable manifolds connecting points $x\in\mathbb{T}^2/A^u_f$ to $A^u_f$ is bounded below for $f\in\mathcal{U}$, $\rho$ can be chosen to have uniformly bounded $C^{1+\alpha}$ norm, and hence $\rho\circ h$ has uniformly bounded $\alpha_0$-seminorm. Since $D_i$ is uniformly bounded, it remains to estimate $|D_1(x)-D_1(x')|$. Since $x$ and $x'$ are on the same stable manifold, $y(x)=y(x')$. Then by (\ref{eqn:Derivative of Holonomy Formula}), we have
    $$
    |\log D_u\hol^{s,f}_{a,b}(x)-\log D_u\hol^{s,f}_{a',b}(x')|\leq \sum_{i=0}^\infty |\log D_uf(f^i(x))-\log D_uf(f^i(x'))|$$
    $$\leq \frac{1}{\mu_1}d^s_f(x,x')\sum_{i=0}^\infty \nu_0^i=\frac{1}{\mu_1(1-\nu_0)}d^s_f(x,x').
    $$
    Putting this together the definitions of the $D_i$ (\ref{eqn:Expression for Derivative from above}) and (\ref{eqn:Transverse regularity First Step of Construction}), we have
    $$
    |D_uh_N(x)-D_uh_N(x')|\leq Cd^s_f(x,x')^{\alpha_0},
    $$
    with $C>0$ uniform.
\end{proof}

\begin{proof}[Proof of Lemma \ref{lem:s-diff of second step}]
    The proof is largely analogous to that of Lemma \ref{lem:Uniform bound for unstable derivative first step} except we need to establish uniform $C^{1+\alpha}$ bounds on $\Tilde{\hol}^{u,g}$. To be precise, it suffices to prove that for $x'\in \mathbb{T}^2/A^s_f$
    \begin{equation}
    \label{eqn: Regularity of Fake Unstable Holonomy}
    \Tilde{\hol}^{u,g}_{h(\overline{y}(x')),h_N(x')}:A^s_g\rightarrow \Tilde{W}^s_g(h_N(x'))   
    \end{equation}    
    is uniformly $C^{1+\alpha}$. However, this is in essence the content of Lemma \ref{lem: Fake Stable Foliation}. Indeed, by Lemma \ref{lem: Fake Stable Foliation}, we can view $\Tilde{W}^s_g(h_N(x'))$ as the graph of the $C^{1+\alpha}$ function $\rho$ over some stable manifold of $g$, which, after a composition with a $C^{1+\alpha}$ unstable holonomy $\hol^{u,g}$, we may assume is $A^u_g$. By analogous computations as in Lemma \ref{lem:Uniform bound for unstable derivative first step}, we have that the unstable holonomy $\hol^{u,g}$ is uniformly $C^{1+\alpha}$, and since the likewise we have that the $\rho$ have uniformly bounded $C^{1+\alpha}$ norm, we have that (\ref{eqn: Regularity of Fake Unstable Holonomy}) is uniformly $C^{1+\alpha}_s$. Since $\overline{\rho}$ is constant along leaves of $\mathcal{F}^{s,g}$, it follows that $\overline{h}_N|_{W^s_f(x)}$ is uniformly $C^{1+\alpha}$ for every stable leaf of $f$. 
\end{proof}

\begin{proof}[Proof of Lemma \ref{lem: u-diff Second Step}]
    We first observe that by the same proof as Lemma \ref{lem: Fake Stable Foliation}, $\overline{h}_N$ sends the unstable foliation $\mathcal{F}^{u,f}$ to a $C^{1+\alpha}$-foliation consisting of $C^2$ leaves which we denote by $\Tilde{\mathcal{F}}^{u,g}$. While $h_N$ traces out the $C^2$ leaves of $\Tilde{\mathcal{F}}^{s,g}$ with a H\"older parameterization, $\overline{h}_N$ traces out the leaves of $\Tilde{\mathcal{F}}^{u,g}$ with a $C^{1+\alpha}$ parameterization since $h_N\in C^{1+\alpha}_u$ by Lemma \ref{lem:Uniform bound for unstable derivative first step}. Since $h_N$ is uniformly $C^{1+\alpha}$ and $\overline{\rho}$ can be chosen uniformly $C^{1+\alpha}$ for all $f,g\in\mathcal{U}$, it follows that $\overline{h}_N|_{W^u_f(x)}$ is uniformly $C^{1+\alpha}$ for every unstable leaf of $f$.
\end{proof}

\section{Proof of Main Theorem: The $C^0$ Estimates}
\label{sec: Proof the Main Theorem the C^0 Estimates}
The main goal of this section is to prove the following part of Theorem \ref{Thm:Main Theorem}:
\begin{theorem}
\label{thm: C^0 Estimate Theorem}
    There exists constants $C_0>0$, $0<\lambda_0<1$ such that for any $f,g\in\mathcal{U}$, and any conjugacy $h$ in the homotopy class of the identity and $N\in\mathbb{N}$ as in Theorem \ref{Thm:Main Theorem}, 
    $$
    d_{C^0}(h,\overline{h}_N)<C_0\lambda_0^N,
    $$
    where $\overline{h}_N$ is the conjugacy constructed in Section 3.
\end{theorem}
The first step will be to break up estimate using the intermediate conjugacy $h_N$ constructed in \ref{def:h_N}:
\begin{equation}
\label{eqn:First Reduction Estimate}
     d_{C^0}(h,\overline{h}_N)\leq  d_{C^0}(h,h_N)+ d_{C^0}(h_N,\overline{h}_N).
\end{equation}
The two terms on the right side of (\ref{eqn:First Reduction Estimate}) will be handled identically and for concreteness we will focus on explaining how to estimate the first term.

To estimate  $d_{C^0}(h,h_N)$, we will begin by showing how to reduce to estimating the pointwise distance $d(h(x),h_N(x))<C_0\lambda_0^N$ for $x\in A^u_f$. The main goal of this section will be to prove the following lemma:
\begin{lemma}
\label{lem:First Step Estimate}
    There exists constants $C_0>0$, $0<\lambda_0<1$ such that for any $f,g\in\mathcal{U}$, and any conjugacy $h$ in the homotopy class of the identity and $N\in\mathbb{N}$ as in Theorem \ref{Thm:Main Theorem}, 
    $$
    d(h(x),h_N(x))<C_0\lambda_0^N,
    $$
    for $x\in A^u_f$, where $h_N$ is the conjugacy constructed in Section 3.
\end{lemma}
First, let us see how Theorem \ref{thm: C^0 Estimate Theorem} follows from Lemma \ref{lem:First Step Estimate}:
\begin{proof}[Proof of Theorem \ref{thm: C^0 Estimate Theorem} Assuming Lemma \ref{lem:First Step Estimate}]
    To begin, fix a point $x\in\mathbb{T}^2/A^u_f$, and observe that 
$d(h(x),h_N(x))\leq d^u_g(h(x),h_N(x))$, where $d_g^u$ denotes the induced distance in the leaf $W^u_g(h(x))$.
Now consider the two points $y,z\in A^u_f$ obtained by taking the stable holonomies of $x$ in either direction. Then by \ref{def:h_N}, $h_N(x)$ lies between points $\hol^{s,g}_{h(y),h(x)}(h_N(y))$ and $\hol^{s,g}_{h(z),h(x)}(h_N(z))$ in the leaf $W^u_g(h(x))$. 
Therefore, 
$$
d^u_g(h(x),h_N(x))\leq\max\{d_g^u(h(x),\hol^{s,g}_{h(y),h(x)}(h_N(y))),d_g^u(h(x),\hol^{s,g}_{h(z),h(x)}(h_N(z)))\}=$$
$$
\max\{d_g^u(\hol^{s,g}_{h(y),h(x)}(h(y)),\hol^{s,g}_{h(y),h(x)}(h_N(y))),d_g^u(\hol^{s,g}_{h(z),h(x)}(h(z)),\hol^{s,g}_{h(z),h(x)}(h_N(z)))\},
$$
 Since $\hol^{s,g}_{h(y),h(x)}$ and $\hol^{s,g}_{h(z),h(x)}$ are Lipschitz continuous functions with respect to the induced leaf-distances, we have
$$
\max\{d^u(\hol^{s,g}_{h(y),h(x)}(h(y)),\hol^{s,g}_{h(y),h(x)}(h_N(y))),d^u(\hol^{s,g}_{h(z),h(x)}(h(z)),\hol^{s,g}_{h(z),h(x)}(h_N(z)))\}\leq$$
$$
C_{Lip}\max\{d_g^u(h(y),h_N(y)),d_g^u(h(z),h_N(z))\}.
$$
It remains to bound the unstable distance $d_g^u$ on $A^u_g$ by the standard Riemannian distance $d$. To be precise, it will be sufficient to show that there exists $C^u_g>0$ such that $d_g^u(a,b)\leq C^u_gd(a,b)$ for all $a,b\in A^u_g$ such that $d(a,b)\leq \delta$, where $\delta>0$ comes from the local product structure constants in Theorem \ref{thm: Local Product Structure Uniform}. Then we have
$$
d(h(x),h_N(x))<C_{Lip}C^u_gC_0\lambda_0^N,
$$
whenever $N$ is sufficiently large that $C_0\lambda_0^N\leq \delta$ in Lemma \ref{lem:First Step Estimate} below. For small $N$, we have the trivial estimate 
$$d(h(x),h_N(x))\leq 2\diam(\mathbb{T}^2)=\frac{2\diam(\mathbb{T}^2)}{\lambda_0^{N}}\lambda_0^{N}.$$

\begin{sublemma}
\label{lem: Uniform change in leaf metric}
    The constants $C_{Lip},C^u_g>0$ can be made uniform in $g\in\mathcal{U}$.
\end{sublemma}
\begin{proof}
    The uniformity of $C_{Lip}$ follows from the proof of Lemma \ref{lem: Uniformity along A^u_f}. The following argument for uniformity of $C^u_g$ is well-known to experts but we include it here for the sake of completeness.

    Given $x\in A^u_g$, there exists a line field $L$ and $\theta<\pi/2$ such that $\measuredangle(E^u_g(y),L(y))\leq \theta$ for every $y\in B_\delta(x)$. Moreover, the same is true for every $\Tilde{g}$ sufficiently $C^1$-close to $g$ with the same line field $L$ and angle $\theta$. Therefore we can locally consider $A^u_g$ as a Lipschitz graph over the line field $L$ with derivative bounded by $\tan(\theta)$.  Thus, for $y,z\in A^u_g\cap B_\delta(x)$, we have $d^u_g(y,z)\leq \tan(\theta)d(y,z).$ Since $A^u_g$ can be covered by finitely many $\delta$-balls (with the same number of balls for every $\Tilde{g}$ in a $C^1$-neighborhood of $g$) we are finished locally. The lemma is then finished by applying the Arzel\`a-Ascoli argument.    
\end{proof}
    In light of this sublemma, we can absorb the constants $C_{Lip}$ and $C^u_g$ into $C_0$, and the reduction is complete.
\end{proof}

\begin{proof}[Proof of Lemma \ref{lem:First Step Estimate}]
Up until now, we have suppressed the dependence of $h_N$ on the reparameterizations $\sigma_i$ in our notation. From now on, $h_N$ will denote the piecewise $C^2$ map defined by \ref{def:h_N|A^u_{f,1}} and $h_{N,\sigma}$ will denote the map obtained by composing $h_N$ by the $\sigma_i$ on each interval $A^u_{f,i}$. Recall that we take $\sigma$ to satisfy $|\sigma|_{C^0}\leq \gamma^{2N}$, where $0<\gamma<1$ is from Sublemma \ref{lem: C1 size of reparam}. Therefore we have $d(h_N(x),h_{N,\sigma}(x))<\gamma^{2N}$. By the triangle inequality, it suffices to estimate
$d(h(x),h_N(x))$ for $x\in A^u_{f,i}$, $i=1,2,3.$

Fix $x\in A^u_{f,1}$. Then
$$
d(h(x),h_N(x))=|h(x)-h_N(x)|=| (I^u_{g,h(x_0)})^{-1}(I^u_{g,h(x_0)}(h(x)))- (I^u_{g,h(x_0)})^{-1}(I^u_{f,x_0}(x))|
$$
$$
\leq \lip((I^u_{g,h(x_0)}|_{A^u_{g,1}})^{-1})|I^u_{g,h(x_0)}(h(x))-I^u_{f,x_0}(x)|,
$$
where $\lip(\varphi)$ denotes the Lipschitz constant of the function $\varphi$. 
\begin{sublemma}
    There exists $L>0$ such that for every $f,g\in\mathcal{U}$ and every conjugacy $h$ homotopic to the identity such that $hf=gh,$ $\lip((I^u_{g,h(x_0)}|_{A^u_{g,1}})^{-1})\leq L$.
\end{sublemma}
\begin{proof}
    Since $(I^u_{g,h(x_0)}|_{A^u_{g,1}})^{-1})'=\frac{1}{\omega^u_g(I^u_{g,h(x_0)}|_{A^u_{g,1}})^{-1})}$, the uniformity claim is contained in the proof of Lemma \ref{lem: Uniformity along A^u_f}.
\end{proof}
We can thus absorb the constant $L$ into $C_0$. We therefore must estimate the difference
\begin{equation}
\label{eqn:C^0 Estimate}
   \left| \int_{h(x_0)}^{h(x)}d\mu^u_{g,h(x_0)}-\int_{x_0}^{x}d\mu^u_{f,x_0}\right| 
\end{equation}
This will be done in two main steps. First, we will show how to replace the integrals with respect to conditional measures with integrals with respect to the respective SRB measures over thin rectangles. Next, we will use an effective equidistribution theorem to estimate the difference between the integrals with respect to SRB measure using the matching periodic data.

Given a point $x\in A^u_{f,1}$ and $\varepsilon>0$, let
$A^u_{f,1}(x,\varepsilon)$ denote a su-rectangle constructed as follows: Recall that $W^s_\varepsilon(x_0)$ denotes the local stable manifold of $x_0$ of size $\varepsilon$. We then form $A^u_{f,1}(x,\varepsilon)$ by sliding $W^s_{f,\varepsilon}(x_0)$ along the unstable holonomy $\hol^{u,f}_{x_0,x}$. Then for every $x\in A^u_{f,1}$ and all $\varepsilon>0$ small enough, $A^u_{f,1}(x,\varepsilon)$ is an embedded rectangle.

\begin{theorem}
\label{Thm:Effective Estimate of Conditional Measures}
    Let $\varphi:\mathbb{T}^2\rightarrow\mathbb{R}$ be a Lipschitz continuous function. Then there exists a constant $C_2>0$ such that for every $\varepsilon>0$, 
    \begin{equation}
       \left| \int_{x_0}^x\varphi d\mu^u_{f,x_0}-\frac{1}{\mu_f(A^u_{f,1}(x_0',\varepsilon))}\int_{A^u_{f,1}(x,\varepsilon)}\varphi d\mu_f\right|\leq C_2||\varphi||_{Lip}\varepsilon.
    \end{equation}
\end{theorem}
\begin{proof}
 Let 
 $$
 d\mu_f|_{A^u_{f,1}(x_0',\varepsilon)}=\Psi_{f,x_0}(u,s)d\mu^u_{f,x_0}(u)d\hat{\mu}_f(s)
 $$
be the Lipschitz local product structure of $\mu_f$ on the rectangle $A^u_{f,1}(x_0',\varepsilon)$ where $\hat{\mu}_f$ is the quotient measure of the rectangle $A^u_{f,1}(x_0',\varepsilon)$ on the the transversal $W^s_{f,\varepsilon}(x_0)$. Here we are using stable and unstable coordinates $s$ and $u$, respectively. In these coordinates, $s=0$ corresponds to the unstable curve $A^u_{f,1}$, and we have $\Psi_{f,x_0}(u,0)=1$ for all $u$. Recalling that $\hat{\mu}_f(W^s_{f,\varepsilon}(x_0))=\mu_f(A^u_{f,1}(x_0',\varepsilon))$, we can write
$$
\int_{x_0}^x\varphi d\mu^u_{f,x_0}=\frac{\mu_f(A^u_{f,1}(x_0',\varepsilon))}{\mu_f(A^u_{f,1}(x_0',\varepsilon))}\int_{x_0}^x\varphi d\mu^u_{f,x_0}=
\frac{1}{\mu_f(A^u_{f,1}(x_0',\varepsilon))}\int_{W^s_{f,\varepsilon}(x_0)}\int^x_{x_0}\varphi(u,0)d\mu^u_{f,x_0}d\hat{\mu}_f(s)=$$
$$
\frac{1}{\mu_f(A^u_{f,1}(x_0',\varepsilon))}\int_{W^s_{f,\varepsilon}(x_0)}\int^x_{x_0}\varphi(u,0)\Psi_{f,x_0}(u,0) d\mu^u_{f,x_0}d\hat{\mu}_f(s).
$$
Likewise,
$$
\frac{1}{\mu_f(A^u_{f,1}(x_0',\varepsilon))}\int_{A^u_{f,1}(x,\varepsilon)}\varphi d\mu_f=
\frac{1}{\mu_f(A^u_{f,1}(x_0',\varepsilon))}\int_{W^s_{f,\varepsilon}(x_0)}\int^x_{x_0}\varphi(u,s)\Psi_{f,x_0}(u,s)d\mu^u_{f,x_0}(u)d\hat{\mu}_f(s).
$$
Therefore,
$$
 \left| \int_{x_0}^x\varphi d\mu^u_{f,x_0}-\frac{1}{\mu_f(A^u_{f,1}(x_0',\varepsilon))}\int_{A^u_{f,1}(x,\varepsilon)}\varphi d\mu_f\right|\leq
$$
$$
\frac{1}{\mu_f(A^u_{f,1}(x_0',\varepsilon))}\int_{W^s_{f,\varepsilon}(x_0)}\int^x_{x_0}\left| \varphi(u,s)\Psi_{f,x_0}(u,s)-\varphi(u,0)\Psi_{f,x_0}(u,0)\right|d\mu^u_{f,x_0}(u)d\hat{\mu}_f(s)\leq
$$
$$
\frac{1}{\mu_f(A^u_{f,1}(x_0',\varepsilon))}\int_{W^s_{f,\varepsilon}(x_0)}\int^x_{x_0}|\varphi(u,s)|\left|\Psi_{f,x_0}(u,s)-\Psi_{f,x_0}(u,0)\right|d\mu^u_{f,x_0}(u)d\hat{\mu}_f(s)+
$$
$$
\frac{1}{\mu_f(A^u_{f,1}(x_0',\varepsilon))}\int_{W^s_{f,\varepsilon}(x_0)}\int^x_{x_0}\Psi_{f,x_0}(u,0)\left|\varphi(u,s)-\varphi(u,0)\right|d\mu^u_{f,x_0}(u)d\hat{\mu}_f(s)\leq
$$
$$
C(x_0,x)||\varphi||_{\infty}||\Psi_{f,x_0}|_{Lip} d((u,s),(u,0))+C(x_0,x)|\varphi|_{Lip}d((u,s),(u,0)),
$$
where $C(x_0,x)=\int_{x_0}^xd\mu^u_{f,x_0}<C(x_0,x_0')$, which is uniformly bounded in $f\in\mathcal{U}$ by the standard Arzel\`a-Ascoli argument. Finally, we claim that $d((u,s),(u,0))<C_2\epsilon$ for some uniform constant $C_2>0$. To see this, first observe that 
$d((u,s),(u,0))=d(\hol^{u,f}(0,s),\hol^{u,f}(0,0))$. Since we can always bound the metric $d$ on $\mathbb{T}^2$ by the leaf metric on the stable manifold, we have
$$d(\hol^{u,f}(0,s),\hol^{u,f}(0,0))\leq d^s(\hol^{u,f}(0,s),\hol^{u,f}(0,0))\leq |\hol^{u,f}|_{Lip}d^s((0,s),(0,0))<C_{Lip}\varepsilon.$$ 
Absorbing all constants into $C_2$ ends the proof.    
\end{proof}
\noindent Analogous statements hold for $g$ as well as for rectangles based on the unstable segments $A^u_{f,2}$ and $A^u_{f,3}$.

Given $n\in\mathbb{N}$, let 
$$
\mu_f^n=\frac{1}{Z_n(f)}\sum_{x\in\fix{f^n}}\frac{1}{D_u(f^n)(x)}\delta_x,
$$
where $Z_n(f)$ is a normalization constant to make $\mu_f^n$ a probability measure. It follows from Bowen's equidistribution theorem (see \cite{Bowen_1974}) that these discrete measures converge (in the weak$^*$-topology) to the SRB measure $\mu_f$. The next theorem makes this quantitative.

\begin{theorem}
\label{Thm:Effective Equidistribution}
    Let $\mathcal{U}$ be as in Theorem \ref{Thm:Main Theorem}. Then there exists constants $C_3>0$ and $0<\tau<1$ depending only on $\mathcal{U}$ such that for every $f\in\mathcal{U}$, every $n\in\mathbb{N}$, and every Lipschitz function $\varphi:\mathbb{T}^2\rightarrow\mathbb{R}$, we have 
\begin{equation}
    \left|\int \varphi d\mu_f-\int \varphi d\mu_f^n  \right|\leq 
    C_3||\varphi||_{Lip}\tau^n.
\end{equation}
   
\end{theorem}
We defer the proof of Theorem \ref{Thm:Effective Equidistribution} to section 6.

We will now explain how to use Theorems \ref{Thm:Effective Estimate of Conditional Measures} and \ref{Thm:Effective Equidistribution} to estimate (\ref{eqn:C^0 Estimate}). First observe that $h(A^u_{f,1}(x,\varepsilon))$ is an su-rectangle containing the segment $[h(x_0),h(x)]_u$. Therefore, we may write
$h(A^u_{f,1}(x,\varepsilon)):=A^u_{g,1}(h(x),\Tilde{\varepsilon})$, where, by H\"older continuity of $h$, we have
$$|h^{-1}|_{\alpha_0}\varepsilon^{1/\alpha_0}\leq
\Tilde{\varepsilon}\leq |h|_{\alpha_0}\varepsilon^{\alpha_0}.$$
Here, $\alpha_0$ is the H\"older exponent of $h$, which is uniform by Lemma \ref{lem:Uniform Holder}. 

We now estimate
$$
\left| \int_{h(x_0)}^{h(x)}d\mu^u_{g,h(x_0)}-\int_{x_0}^{x}d\mu^u_{f,x_0}\right|$$
$$
\leq
 \left| \int_{x_0}^xd\mu^u_{f,x_0}-\frac{1}{\mu_f(A^u_{f,1}(x_0',\varepsilon))}\int_{A^u_{f,1}(x,\varepsilon)} d\mu_f\right|+$$
$$
  \left| \int_{h(x_0)}^{h(x)} d\mu^u_{g,h(x_0)}-\frac{1}{\mu_g(A^u_{g,1}(h(x_0'),\Tilde{\varepsilon}))}\int_{A^u_{g,1}(h(x),\Tilde{\varepsilon})} d\mu_g\right|+
$$
$$
\left| \frac{1}{\mu_f(A^u_{f,1}(x_0',\varepsilon))}\int_{A^u_{f,1}(x,\varepsilon)} d\mu_f-\frac{1}{\mu_g(A^u_{g,1}(h(x_0'),\Tilde{\varepsilon}))}\int_{A^u_{g,1}(h(x),\Tilde{\varepsilon})}d\mu_g\right|.
$$
We now choose $\varepsilon=\tau^{N/4}$, where $0<\tau<1$ is from Theorem \ref{Thm:Effective Equidistribution} and $N\in\mathbb{N}$ is from Theorem \ref{Thm:Main Theorem}. Then applying Theorem \ref{Thm:Effective Estimate of Conditional Measures} to the first two terms above, with $\varphi=1$ the constant function, we get that these terms are bounded by $C_2\tau^{N/4}$ and $C_2C_{\alpha_0}\tau^{\alpha_0N/4}$, respectively. We further split up the third term as follows:
$$
\left| \frac{1}{\mu_f(A^u_{f,1}(x_0',\varepsilon))}\int_{A^u_{f,1}(x,\varepsilon)} d\mu_f-\frac{1}{\mu_g(A^u_{g,1}(h(x_0'),\Tilde{\varepsilon}))}\int_{A^u_{g,1}(h(x),\Tilde{\varepsilon})}d\mu_g\right|\leq
$$
$$
\left| \frac{1}{\mu_f(A^u_{f,1}(x_0',\varepsilon))}\int_{A^u_{f,1}(x,\varepsilon)} d\mu_f-\frac{1}{\mu_f(A^u_{f,1}(x_0',\varepsilon))}\int_{A^u_{g,1}(h(x),\Tilde{\varepsilon})}d\mu_g\right|+
$$
$$
\left|\frac{1}{\mu_f(A^u_{f,1}(x_0',\varepsilon))}\int_{A^u_{g,1}(h(x),\Tilde{\varepsilon})}d\mu_g-\frac{1}{\mu_g(A^u_{g,1}(h(x_0'),\Tilde{\varepsilon}))}\int_{A^u_{g,1}(h(x),\Tilde{\varepsilon})}d\mu_g\right|=
$$
$$
\frac{1}{\mu_f(A^u_{f,1}(x_0',\varepsilon))} \left|\int_{A^u_{f,1}(x,\varepsilon)} d\mu_f-\int_{A^u_{g,1}(h(x),\Tilde{\varepsilon})}d\mu_g
\right|+$$
$$\left|\frac{1}{\mu_f(A^u_{f,1}(x_0',\varepsilon))}- \frac{1}{\mu_g(A^u_{g,1}(h(x_0'),\Tilde{\varepsilon}))}\right|\int_{A^u_{g,1}(h(x),\Tilde{\varepsilon})}d\mu_g\leq
$$
$$
\frac{1}{\mu_f(A^u_{f,1}(x_0',\varepsilon))} \left|\int_{A^u_{f,1}(x,\varepsilon)} d\mu_f-\int_{A^u_{g,1}(h(x),\Tilde{\varepsilon})}d\mu_g
\right|+$$
$$
\left| \frac{1}{\mu_f(A^u_{f,1}(x_0',\varepsilon))}- \frac{1}{\mu_g(A^u_{g,1}(h(x_0'),\Tilde{\varepsilon}))}\right|\mu_g(A^u_{g,1}(h(x_0'),\Tilde{\varepsilon}))=
$$
$$
\frac{1}{\mu_f(A^u_{f,1}(x_0',\varepsilon))} \left|\int_{A^u_{f,1}(x,\varepsilon)} d\mu_f-\int_{A^u_{g,1}(h(x),\Tilde{\varepsilon})}d\mu_g
\right|+$$
$$
\frac{1}{\mu_f(A^u_{f,1}(x_0',\varepsilon))} \left|\mu_f(A^u_{f,1}(x_0',\varepsilon))- \mu_g(A^u_{g,1}(h(x_0'),\Tilde{\varepsilon}))\right|\leq
$$
$$
\frac{1}{\mu_f(A^u_{f,1}(x_0',\varepsilon))} \left|\int_{A^u_{f,1}(x,\varepsilon)} d\mu_f-\int_{A^u_{g,1}(h(x),\Tilde{\varepsilon})}d\mu_g
\right|+
$$
\begin{equation}
\label{eqn:Pre-EE1}
\frac{1}{\mu_f(A^u_{f,1}(x_0',\varepsilon))} \left|\int_{A^u_{f,1}(x_0',\varepsilon)} d\mu_f-\int_{A^u_{g,1}(h(x_0'),\Tilde{\varepsilon})}d\mu_g
\right|.
\end{equation}
We will handle both terms on the right side of (\ref{eqn:Pre-EE1}) in exactly the same way using Theorem \ref{Thm:Effective Equidistribution}. However, before we can do so we need to show that the measures $\mu_f(A^u_{f,1}(x_0',\varepsilon))$ are on a comparable size scale to $\varepsilon$. This is the first place where we will use the area-preserving hypothesis on $f$.
\begin{sublemma}
    \label{lem: Comparing area to Length}
    There exists a constant $C>0$ such that for every $f\in\mathcal{U}$, we have $C^{-1}\varepsilon\leq \mu_f(A^u_{f,1}(x_0',\varepsilon))\leq C\varepsilon$.
\end{sublemma}
\begin{proof}
    First notice that by (\ref{eqn: Length Upper bound}) and (\ref{eqn: Length Lower Bound}), we have uniform upper and lower bounds on the lengths of the segments $A^u_{f,1}$ for all $f\in\mathcal{U}$. Since the stable and unstable holonomy maps are uniformly Lipschitz, we have uniform upper and lower bounds on the lengths of all stable and unstable transversals in $A^u_{f,1}(x_0',\varepsilon)$. Then since $\mu_f$ is equivalent to area with a density uniformly bounded above and below by Lemma \ref{lem:Uniform bounds on area Density}, we can uniformly compare $\mu_f(A^u_{f,1}(x_0',\varepsilon))$ to the product of the lengths of $A^u_{f,1}$ and $W^s_{f,\varepsilon}(x_0)$. That is, we have some uniform constant $C>0$ such that 
    $$
    C^{-1}d^u_f(x_0,x_0')\varepsilon\leq \mu_f(A^u_{f,1}(x_0',\varepsilon))\leq Cd^u_f(x_0,x_0')\varepsilon.
    $$
    By (\ref{eqn: Length Upper bound}) and (\ref{eqn: Length Lower Bound}), we may absorb the length $d^u_f(x_0,x_0')$ into $C$ without losing uniformity.
    \end{proof}

In light of Lemma \ref{lem: Comparing area to Length}, it is, up to a constant, sufficient to estimate
\begin{equation}
    \label{eqn: Pre-EE2}
    \frac{1}{\tau^{N/4}} \left|\int_{A^u_{f,1}(x_0',\varepsilon)} d\mu_f-\int_{A^u_{g,1}(h(x_0'),\Tilde{\varepsilon})}d\mu_g
\right|.
\end{equation}

Since $D_uf^N(p)=D_ug^N(h(p))$ for every $p\in\fix(f^N)$, we have $\mu_f^N=h^*\mu_g^N$. Moreover, since by definition $A^u_{g,1}(h(x),\Tilde{\varepsilon})=h(A^u_{f,1}(x,\varepsilon))$, we have
$$
\int_{A^u_{f,1}(x,\varepsilon)}d\mu_f^N=\int_{A^u_{g,1}(h(x),\Tilde{\varepsilon})}d\mu_g^N.
$$
Therefore,
$$
\left|\int_{A^u_{f,1}(x,\varepsilon)} d\mu_f-\int_{A^u_{g,1}(h(x),\Tilde{\varepsilon})}d\mu_g
\right|\leq
$$
\begin{equation}
    \label{eqn:Pre-EE2}
\left|\int_{A^u_{f,1}(x,\varepsilon)} d\mu_f-\int_{A^u_{f,1}(x,\varepsilon)}d\mu_f^N
\right|+
\left|\int_{A^u_{g,1}(h(x),\Tilde{\varepsilon})}d\mu_g-\int_{A^u_{g,1}(h(x),\Tilde{\varepsilon})}d\mu_g^N\right|.    
\end{equation}
We would like to apply Theorem \ref{Thm:Effective Equidistribution} to each term on the right side of (\ref{eqn:Pre-EE2}). However, we can not do this directly since the characteristic functions $\chi_{A^u_{f,1}(x,\varepsilon)}$ and 
$\chi_{A^u_{g,1}(h(x),\Tilde{\varepsilon})}$ are not Lipschitz continuous. We will instead first approximate the characteristic functions by Lipschitz function and then apply Theorem \ref{Thm:Effective Equidistribution} to these function. Here we will use the area-preserving hypothesis on both $f$ and $g$.

\begin{sublemma}
    \label{lem: Estimate Characteristic Function by Lipschitz Functions}
    There exists a one-parameter family of Lipschitz functions $\varphi^t_{f,x,\varepsilon}:\mathbb{T}^2\rightarrow [0,1]$ satisfying the following properties:
    \begin{enumerate}
        \item The family $\varphi^t_{f,x,\varepsilon}$ varies continuously with $t$ in the $C^0$-topology;
        \item $\varphi^0_{f,x,\varepsilon}\leq \chi_{A^u_{f,1}(x,\varepsilon)}$ and $\varphi^1_{f,x,\varepsilon}\geq \chi_{A^u_{f,1}(x,\varepsilon)}$;
        \item For every $t\in[0,1]$, $|\varphi^t_{f,x,\varepsilon}|_{Lip}\leq \tau^{-N/2}$;
        \item For every $t\in[0,1]$, $|\varphi^s_{f,x,\varepsilon}-\chi_{A^u_{f,1}(x,\varepsilon)}|=0$ except on a set $\Omega^f_N$ of measure $\mu_f(\Omega^f_N)\leq C_4\tau^{N/2}$, where $C_4>0$ is uniformly bounded for $f\in\mathcal{U}$.
\end{enumerate}
\end{sublemma}
\begin{proof}
    We begin by constructing $\varphi^1_{f,x,\varepsilon}$ as being equal to $1$ on the set $A^u_{f,1}(x,\varepsilon)$. Then on the unstable manifold boundary $\partial^uA^u_{f,1}(x,\varepsilon)$, $\varphi^1_{f,x,\varepsilon}$ will decrease to $0$ smoothly along the unstable manifold at a rate bounded by $\tau^{-N/2}$; likewise on $\partial^sA^u_{f,1}(x,\varepsilon)$, $\varphi^1_{f,x,\varepsilon}$ will decrease to $0$ smoothly along the stable manifold at a rate bounded by $\tau^{-N/2}$. Then for continuity at the corner points we extend $\varphi^1_{f,x,\varepsilon}$ in any way as long as the slope is bounded by $\tau^{-N/2}$. Next, from $t=1$ to $t=\frac{1}{2}$, we continuously contract $\varphi^1_{f,x,\varepsilon}$ along the stable manifolds so that for $\varphi^{1/2}_{f,x,\varepsilon}$, the support of $\varphi^{1/2}_{f,x,\varepsilon}$ is entirely inside of  $A^u_{f,1}(x,\varepsilon)$, except for on the ``top" and ``bottom" unstable segments. Then from $t=\frac{1}{2}$ to $t=0$, we continuously contract along the unstable manifolds until the support of $\varphi^{0}_{f,x,\varepsilon}$ is contained entirely inside of $A^u_{f,1}(x,\varepsilon)$. Constructed in this way, the family $\varphi^t_{f,x\varepsilon}$ clearly satisfies $1$, $2$, and $3$.

    For condition $4$, we use Lemma \ref{lem:Uniform bounds on area Density} (which in particular uses our area-preserving assumption) to reduce the problem to estimating the Lebesgue measure of the set $\Omega^f_N$. For all $t\in [0,1]$, $\Omega^f_N$ is contained in a union of four (partially overlapping) rectangles. Estimating the measures of these rectangles is very similar to the proof of Lemma \ref{lem: Comparing area to Length}. The ``top" and ``bottom" rectangles have dimensions both bounded by $\tau^{N/2}$. The side rectangles have a long length bounded by the length of $A^u_{f,1}$ (up to some uniform constant), and width bounded by $\tau^{N/2}$. Adding up these measures, we get the claimed bound.
\end{proof}

Using this family, we can estimate
$$
\left|\int_{A^u_{f,1}(x,\varepsilon)} d\mu_f-\int_{A^u_{f,1}(x,\varepsilon)}d\mu_f^N
\right|\leq
$$
$$
\left|\int_{A^u_{f,1}(x,\varepsilon)} d\mu_f-\int \varphi^t_{f,x,\varepsilon}d\mu_f\right|+\left|\int \varphi^t_{f,x,\varepsilon}d\mu_f-
\int \varphi^t_{f,x,\varepsilon}d\mu_f^N\right|+
\left|\int \varphi^t_{f,x,\varepsilon}d\mu_f^N-\int_{A^u_{f,1}(x,\varepsilon)}d\mu_f^N \right|.
$$
We consider each of these terms separately. First
$$
\left|\int_{A^u_{f,1}(x,\varepsilon)} d\mu_f-\int \varphi^t_{f,x,\varepsilon}d\mu_f\right|\leq \int\left|\varphi^t_{f,x,\varepsilon}-\chi_{A^u_{f,1}(x,\varepsilon)} \right|d\mu_f=$$
$$
\int_{\Omega^f_N}\left|\varphi^t_{f,x,\varepsilon}-\chi_{A^u_{f,1}(x,\varepsilon)} \right|d\mu_f\leq 2\mu_f(\Omega^f_N)\leq 2C_4\tau^{N/2}.
$$
Next, by Theorem \ref{Thm:Effective Equidistribution},
$$
\left|\int \varphi^t_{f,x,\varepsilon}d\mu_f-
\int \varphi^t_{f,x,\varepsilon}d\mu_f^N\right|\leq C_3||\varphi^t_{f,x,\varepsilon}||_{Lip}\tau^N\leq C_3(1+\tau^{-N/2})\tau^N\leq C_5\tau^{N/2}.
$$
Finally, consider the function
$$
\varphi(t)=\int \varphi^t_{f,x,\varepsilon}d\mu_f^N-\int_{A^u_{f,1}(x,\varepsilon)}d\mu_f^N.
$$
Then $\varphi(t)$ is continuous in $t$ by Lemma \ref{lem: Estimate Characteristic Function by Lipschitz Functions} condition $1$ and by condition 2 of the family 
$\varphi^t_{f,x,\varepsilon}$, $\varphi(0)<0$ and $\varphi(1)>0$. By the intermediate value theorem, there exists some $t_0\in[0,1]$ such that $\varphi(t_0)=0.$ Therefore, using the function $\varphi^{t_0}_{f,x,\varepsilon}$ we have
$$
\frac{1}{\tau^{N/4}} \left|\int_{A^u_{f,1}(x,\varepsilon)} d\mu_f-\int_{A^u_{g,1}(h(x),\Tilde{\varepsilon})}d\mu_g
\right|\leq \frac{1}{\tau^{N/4}}(2C_4\tau^{N/2}+C_5\tau^{N/2})=C_6\tau^{N/4}.
$$
The same argument can be used to estimate
$$
\frac{1}{\tau^{N/4}} \left|\int_{A^u_{f,1}(x_0',\varepsilon)} d\mu_f-\int_{A^u_{g,1}(h(x_0'),\Tilde{\varepsilon})}d\mu_g
\right|$$
except when defining the one-parameter family $\varphi^t_{g,h(x),\Tilde{\varepsilon}}$, we get that 
$|\varphi^t_{g,h(x),\Tilde{\varepsilon}}-\chi_{A^u_{g,1}(h(x),\Tilde{\varepsilon})}|=0$ except on a set $\Omega^g_N$ of measure
$\mu_g(\Omega^g_N)\leq C_7\max\{\tau^{\alpha_0N/4},\tau^{N/4\alpha_0}\}$.

Putting this all together and combining all the constants into $C_0$, we get
$$
d(h(x),h_N(x))\leq C_0\max\{\tau^{\alpha\alpha_0N/4},\tau^{N/4\alpha_0}\}.
$$
This proves Lemma \ref{lem:First Step Estimate} with 
$\lambda_0=\max\{\tau^{\alpha\alpha_0/4},\tau^{1/4\alpha_0}, \gamma^2\}$.
\end{proof}

It remains to get the corresponding bound on $d_{C^0}(h^{-1},\overline{h}_N^{-1})$. As before we break this up as
$d_{C^0}(h^{-1},\overline{h}_N^{-1})\leq d_{C^0}(h^{-1},h_N^{-1})+d_{C^0}(h_N^{-1},\overline{h}_N^{-1})$. These terms will again be handled identically so we will only focus on bounding
$d_{C^0}(h^{-1},h_N^{-1})$. We will use the following simple lemma
\begin{lemma}
    \label{lem:Inverse Estimates}
Let $f,g:\mathbb{R}\rightarrow\mathbb{R}$ be two homeomorphisms and suppose that $g$ is $C^1$. If $|g'(x)|\geq \mu>0$ for every $x\in\mathbb{R}$ and $d_{C^0}(f,g)<\varepsilon$, we have
$d_{C^0}(f^{-1},g^{-1})<\frac{\varepsilon}{\mu}$.    
\end{lemma}
\noindent
The proof of this lemma is an elementary calculation, which we shall skip.

Given $x\in\mathbb{T}^2$, $h_N(x)\in W^u_g(h(x))$, so we may regard $h$ and $h_N$ as functions from $W^u_f(x)$ to $W^u_g(h(x))$, both of which are one-dimensional $C^2$ curves. Note that the arguments in Section \ref{sec: Uniformity of the New Conjugacy} can be repeated nearly verbatim to give uniform lower bounds on $D_uh_N|_{W^u_f(x)}$.  Therefore, we may apply Lemma \ref{lem:Inverse Estimates} to obtain a bound
$$
d_{C^0}(h^{-1},h_N^{-1})< |\min D_uh_N|^{-1}d_{C^0}(h,h_N)<C_0'\lambda_0^N.
$$
Replacing $C_0$ by $\max\{C_0,C_0'\}$, we have the desired bounds.

\section{Proof of Main Theorem: The $C^1$ Estimates}
\label{sec: Proof of the Main Theorem: The C1 Estimates}
Recall that $d_{C^0}(h,\overline{h}_N)\leq C_0\lambda_0^N$. We begin with a simple estimate:
\begin{lemma}
\label{lem:C^0 estimates for F_N}
    Let $f_N:=\overline{h}_N^{-1}g\overline{h}_N$. Then $d_{C^0}(f,f_N)<C_8\lambda_0^N$.
\end{lemma}
\begin{proof}
    $$d_{C^0}(f,f_N)=d_{C^0}(h^{-1}gh,\overline{h}_N^{-1}g\overline{h}_N)\leq
    d_{C^0}(h^{-1}gh,\overline{h}_N^{-1}gh)+d_{C^0}(\overline{h}_N^{-1}gh,\overline{h}_N^{-1}g\overline{h}_N)
    $$
    $$
    \leq d_{C^0}(h^{-1},\overline{h}_N^{-1})+\lip(\overline{h}_N^{-1}g)d_{C^0}(h,\overline{h}_N)<
    (1+\lip(\overline{h}_N^{-1}g))C_0\lambda_0^{N}=C_8\lambda_0^N.
    $$
\end{proof}
Consider the function $F_N:=f-f_N$. Then we have shown that $||F_N||_{C^0}<C_8\lambda_0^N$. The goal for the remainder of this section is to prove $||F_N||_{C^1}<C_1\lambda_1^N$ for some
$C_1>0$ and $0<\lambda_1<1$. We will prove in Lemma \ref{lem: C^1 bound for F_N} below that for $\alpha>0$ as in Theorem \ref{Thm:Uniform Conjugacy Bounds}, $||F_N||_{C^{1+\alpha}}<C_9$, for some uniform constant $C_9>0$. From here the result will follow from the following interpolation theorem (see e.g. \cite{Interpolation_Parabolic}):
\begin{theorem}
\label{Thm:Interpolation}
    For any $0<\epsilon<\alpha$, there exists $0<\theta<1$ and $C_\theta>0$ such that for any $\psi\in C^{1+\alpha}(\mathbb{R}^2)$, 
    $$
    ||\psi||_{C^{1+\epsilon}}\leq C_\theta||\psi||^\theta_{C^0}||\psi||^{1-\theta}_{C^{1+\alpha}}.
    $$
\end{theorem} 
We will use this in conjunction with the following
\begin{lemma}
\label{lem: C^1 bound for F_N}
    There exists $C_9>0$, uniform in $f,g\in\mathcal{U}$ such that
    $||F_N||_{C^{1+\alpha}}<C_9$.
\end{lemma}
Together with Lemma \ref{lem:C^0 estimates for F_N} and Theorem \ref{Thm:Interpolation} we have
$$
||F_N||_{C^1}<||F_N||_{C^{1+\epsilon}}\leq 
C_\theta||F_N||^\theta_{C^0}||F_N||^{1-\theta}_{C^{1+\alpha}}<
C_\theta C_8^\theta C_9^{1-\theta}\lambda_0^{\theta N}.
$$
This proves Theorem \ref{Thm:Main Theorem} with $\lambda_1=\lambda_0^\theta$. It remains to prove Lemma \ref{lem: C^1 bound for F_N}:
\begin{proof}[Proof of Lemma \ref{lem: C^1 bound for F_N}]
    Since $||F_N||_{C^{1+\alpha}}\leq ||f||_{C^{1+\alpha}}+||f_N||_{C^{1+\alpha}}$ and $f$ is contained in the $C^2$ bounded set $\mathcal{U}$, it suffices to uniformly bound $||f_N||_{C^{1+\alpha}}=||\overline{h}_N^{-1}g\overline{h}_N||_{C^{1+\alpha}}$. This follows from the chain rule together with our assumption that $g$ is in the $C^2$ bounded set $\mathcal{U}$ and Theorem \ref{Thm:Uniform Conjugacy Bounds}.
\end{proof}

\section{Proof of Theorem \ref{Thm:Effective Equidistribution}}
\label{sec: Proof of EE}
The proof will proceed similarly to the Proof of Theorem 2.2 in \cite{FDR1}. In particular we start with the following theorem for subshifts of finite type:
\begin{theorem}[Effective Equidistribution for Equilibrium States]
\label{thm:EE for SFT}
Let $(\Sigma_A,\sigma_A)$ be a subshift of finite type, where the transition matrix $A$ is irreducible and aperiodic, and let $\psi\in\mathcal{F}_\theta$ be a Lipschitz continuous potential. Then there exists constants $C_{10}>0$ and $0<\tau<1$ such that for any $\varphi\in\mathcal{F}_\theta$ and all $n\in\mathbb{N}$, 

$$
\bigg|\int \varphi d\mu^n_{\psi}-\int \varphi d\mu_\psi\bigg|\leq
C_{10}||\varphi||_\theta\tau^n,
$$
where $\mu_\psi$ is the unique equilibrium state of $\psi$.
\end{theorem}
See \cite{FDR1} for the proof. Now consider the geometric potential $\psi_f=-\log D_uf$. It is an important characterization of the SRB measure that for $C^2$ Anosov diffeomorphisms, the SRB measure $\mu_f$ is the unique equilibrium state of the potential $\psi_f$. Then using a Markov partition, we can consider a subshift of finite type
$\sigma_A:\Sigma_A\rightarrow\Sigma_A$ that is semi-conjugate to $f:\mathbb{T}^2\rightarrow\mathbb{T}^2$; that is, $\pi_f\circ\sigma_A=f\circ\pi_f$, where $\pi_f:\Sigma_A\rightarrow \mathbb{T}^2$. Given any $0<\theta<1$, we can define a metric on $\Sigma_A$ by
$$
d_{\theta}(x,y)=\theta^{\max\{n\geq0|x_i=y_i, 0\leq |i|<n\}}.
$$
For an appropriately chosen $\theta$ (which depends only on the rate of contraction and regularity of the foliations, both of which are uniform in $\mathcal{U}$), $\pi_f$ is Lipschitz continuous. The idea of the proof then is to consider the lifted potential $\psi_f\circ\pi_f$ and then push the effective equidistribution for the equilibrium state $\mu_{\psi_f\circ\pi_f}$ of $\psi_f\circ\pi_f$ to the desired effective equidistribution of $\mu_f$. 

There are two main difficulties with this. The first is uniformity. The proof of Theorem \ref{thm:EE for SFT} relies heavily on the spectral gap of the transfer operator $\mathcal{L}_{\psi_f\circ\pi_f}$ associated to $\psi_f\circ\pi_f$, and we need $||\mathcal{N}^n_{\psi_f\circ\pi_f}||_\theta<C_{11}\tau^ne^{nP(\psi_f\circ\pi_f)}$ for $C_{11}>0$ and $0<\tau<1$ uniform in $f\in\mathcal{U}$, where $\mathcal{N}^n_{\psi_f\circ\pi_f}$ is the transfer operator $\mathcal{L}_{\psi_f\circ\pi_f}$ minus the projection to the leading (simple) eigenspace. This uniform estimate can be achieved using the Birkhoff cone argument in \cite{Frédéric}.

The second difficulty that arises is the over counting of periodic orbits in the symbolic coding $\sigma_A:\Sigma_A\rightarrow\Sigma_A$. This can be handled by the standard arguments of Manning \cite{Manning74}. Assuming for now uniformity of the constants $C_{10}$ and $\tau$ in Theorem \ref{thm:EE for SFT}, we proof Theorem \ref{Thm:Effective Equidistribution}:
\begin{proof}[Proof of Theorem \ref{Thm:Effective Equidistribution}]
Let $\nu_{\psi_f\circ\pi_f}^n$ be any measure on $\Sigma_A$ such that $\pi_f^*\nu_{\psi_f\circ\pi_f}^n=\mu_{\psi_f}^n$. Then

$$
\left|\int\varphi d\mu_{\psi_f}^n-\int\varphi d\mu_{\psi_f}\right|= 
\left|\int\varphi\circ\pi_f d\nu_{\psi_f\circ\pi_f}^n-\int\varphi\circ\pi_f d\mu_{\psi_f\circ\pi_f}\right|$$
\begin{equation}
\label{eqn:EE SFT 1}
\leq\left|\int\varphi\circ\pi d\mu_{\psi\circ\pi}^n-\int\varphi\circ\pi d\mu_{\psi\circ\pi}\right|+
\left|\int\varphi\circ\pi_f d\nu_{\psi_f\circ\pi_f}^n-\int\varphi\circ\pi_f d\mu_{\psi_f\circ\pi_f}^n\right|.
\end{equation}
The first term in (\ref{eqn:EE SFT 1}) is estimated by Theorem \ref{thm:EE for SFT}. To estimate the second term, we need consider the difference between the measures $\mu_{\psi_f\circ\pi_f}^n$ and $\nu_{\psi_f\circ\pi_f}^n$. In general, they are not the same measures. There are two reasons for this. First, the semi-conjugacy $\pi_f$ is not injective. Rather, it is finite-to-one (for concreteness, say $\pi_f$ is $k$-to-one where $k$ is the size of the Markov partition), and so two or more distinct orbits of period $n$ in $\Sigma_A$ may be mapped to the same orbit of period $n$ in $\mathbb{T}^2$. Second, orbits of period $n$ in $\mathbb{T}^2$ may not lift to orbits of period $n$ in $\Sigma_A$. This happens when a point $x\in\fix(f^n)$ lies on the boundary of two or more rectangles in the Markov Partition. When this happens, $x$ may lift to an orbit or period $2n,3n,\cdots,kn$. Therefore, when choosing the measure
$\nu_{\psi_f\circ\pi_f}^n$, we need to estimate how many points $x\in\fix(\sigma_A^n)$ we have left over that get mapped to duplicate points in $\fix(f^n)$, and how many many points in $\fix(\sigma_A^{jn})$, $j=2,\cdots,k$, that we need to include in $\nu_{\psi_f\circ\pi_f}^n$. We let $A_n$ denote the set of periodic points of $\sigma_A$ over which $\nu_{\psi_f\circ\pi_f}^n$ is defined, and set $B_n=\fix(\sigma_A^n)/A_n$ and $C_n=A_n/\fix(\sigma_A^n)$. Thus, $B_n$ consists of those points of $\fix(\sigma_A^n)$ which are redundant for representing points in $\fix(f^n)$, and $C_n$ consists of those points of period greater than $n$ in $\Sigma_A$ needed to represent points in $\fix(f^n)$.
Then 
$$
\left|\int\varphi\circ\pi_f d\nu_{\psi_f\circ\pi_f}^n-\int\varphi\circ\pi_f d\mu_{\psi_f\circ\pi_f}^n\right|=$$
$$
\left|\frac{1}{Z_n(\psi_f\circ\pi_f)}\sum_{x\in\fix(\sigma_A^n)}e^{S_n\psi_f(\pi_f(x))}\varphi(\pi_f(x))-\frac{1}{Z_n(\psi_f)}\sum_{x\in A_n}e^{S_n\psi_f(\pi_f(x))}\varphi(\pi_f(x))\right|\leq
$$
$$\left| 
\frac{1}{Z_n(\psi_f\circ\pi_f)}-\frac{1}{Z_n(\psi_f)}\right|
\sum_{x\in\fix(\sigma_A^n)}e^{S_n\psi_f(\pi_f(x))}|\varphi(\pi_f(x))| $$
$$
+\frac{1}{Z_n(\psi_f)}\left|\sum_{x\in\fix(\sigma_A^n)}e^{S_n\psi_f(\pi_f(x))}\varphi(\pi_f(x))-
\sum_{x\in A_n}e^{S_n\psi_f(\pi_f(x))}\varphi(\pi_f(x))
\right|\leq
$$
$$
\left|\frac{Z_n(\psi_f\circ\pi_f)}{Z_n(\psi_f)}-1\right|||\varphi||_{\infty}+\frac{1}{Z_n(\psi_f)}\sum_{x\in B_n}e^{S_n\psi_f(\pi_f(x))}|\varphi(\pi_f(x))|+
\frac{1}{Z_n(\psi_f)}\sum_{x\in C_n}e^{S_n\psi_f(\pi_f(x))}|\varphi(\pi_f(x))|.
$$
Notice that 
$$
Z_n(\psi_f\circ\pi_f)=Z_n(\psi_f)+\sum_{x\in B_n}e^{S_n\psi_f(\pi_f(x))},
$$
so that
$$
\left|\frac{Z_n(\psi_f\circ\pi_f)}{Z_n(\psi_f)}-1\right|=
\frac{1}{Z_n(\psi_f)}\sum_{x\in B_n}e^{S_n\psi_f(\pi_f(x))}.
$$
Therefore, it comes down to estimating
$$
\left(\frac{2}{Z_n(\psi_f)}\sum_{x\in B_n}e^{S_n\psi_f(\pi_f(x))}
+\frac{1}{Z_n(\psi_f)}\sum_{x\in C_n}e^{S_n\psi_f(\pi_f(x))}\right)
||\varphi||_\infty.
$$
By Lemma 1 of \cite{Manning74}, the cardinality of the sets $B_n$ and $C_n$ grow at a slower rate than the cardinality of $\fix(f^n)$, and the same is true for the weighted sums over these sets:
$$
\sum_{x\in B_n}e^{S_n\psi_f(\pi_f(x))}\leq D_1\eta^n_1Z_n(\psi_f),
\sum_{x\in C_n}e^{S_n\psi_f(\pi_f(x))}\leq D_2\eta^n_2Z_n(\psi_f),
$$
and moreover the constants $D_1,D_2>0$ and $0<\eta_1,\eta_2<1$ can be made uniform in $f\in\mathcal{U}$. 
\end{proof}

\printbibliography

The Ohio State University, 231 West 18$^{\text{th}}$ Avenue, 43210, Columbus, OH, USA

\textit{Email address}: 
\href{mailto:ohare.26@osu.edu}{ohare.26@osu.edu}

\end{document}